\documentclass[a4paper,11pt,fullpage]{article}
\usepackage[english]{babel}

\usepackage[left=1in,right=1in,top=1in,bottom=1in]{geometry}

\usepackage[usenames]{color}
\usepackage{colortbl}

\usepackage{indentfirst}
\usepackage{amsmath}
\usepackage{amsthm}
\usepackage{amsfonts}
\usepackage{amssymb}
\usepackage{amsxtra}
\usepackage{latexsym}
\usepackage{ifthen}
\usepackage{cite}
\usepackage{esint}
\usepackage[cp1250]{inputenc}

\usepackage{epic}
\usepackage{eepic}

\numberwithin{equation}{section}
\newtheorem{theorem}{Theorem}[section]
\newtheorem{lemma}{Lemma}[section]
\newtheorem{remark}{Remark}[section]

\def\XXint#1#2#3{{\setbox0=\hbox{$#1{#2#3}{\int}$}
     \vcenter{\hbox{$#2#3$}}\kern-.5\wd0}}

\begin{document}

\title{Harnack's inequality for   degenerate   double phase parabolic equations under the non-logarithmic Zhikov's condition}

\author{Mariia Savchenko, Igor Skrypnik, Yevgeniia Yevgenieva}

%\shorttitle{Short paper title for the headers}

%\shortauthor{F. Author, S. Author}

\date{}

\maketitle

\begin{abstract}
We prove Harnack's type inequalities for bounded non-negative solutions of  degenerate parabolic equations with $(p,q)$ growth
$$
u_{t}-{\rm div}\left(\mid \nabla u \mid^{p-2}\nabla u + a(x,t) \mid \nabla u \mid^{q-2}\nabla u \right)=0,\quad a(x,t) \geq 0 ,
$$
under the generalized non-logarithmic Zhikovs conditions
$$ \mid a(x,t)-a(y,\tau)\mid \leqslant A\mu(r) r^{q-p},\quad (x,t),(y,\tau)\in Q_{r,r}(x_{0},t_{0}),$$
$$\lim\limits_{r\rightarrow 0}\mu(r) r^{q-p}=0,\quad \lim\limits_{r\rightarrow 0}\mu(r)=+\infty,\quad \int\limits_{0} \mu^{-\beta}(r)\frac{dr}{r} =+\infty,$$
\noindent with some ~$\beta >0$.
\end{abstract}



\section{Introduction and main results}\label{Introduction}
In this paper we are concerned with a class of parabolic equations with nonstandard growth conditions.
Let $\Omega$ be a domain in $\mathbb{R}^{n},~T>0,~\Omega_{T}:= \Omega \times (0, T).$ We study bounded solutions to the equation
\begin{equation}\label{eq1.1}
u_{t}-\textrm{div}\mathbb{A}(x, t, \nabla u)=0,~(x, t)\in \Omega_{T}.
\end{equation}
We suppose that the functions $\mathbb{A} :\Omega_{T}\times  \mathbb{R}^{n} \rightarrow \mathbb{R}^{n}$ are such that $\mathbb{A}(\cdot, \cdot,  \xi)$ are Lebesgue measurable for all $ \xi \in \mathbb{R}^{n},$ and $\mathbb{A}(x, t, \cdot)$ are continuous for almost all $(x, t)\in \Omega_{T}.$ We also assume that the following structure conditions are satisfied:
\begin{equation}\label{eq1.2}
\begin{aligned}
  \mathbb{A}(x, t, \xi)\xi &\geqslant K_{1} \big( |\xi|^{p} +a(x,t)|\xi|^{q}\big),
   \\
   \mid \mathbb{A}(x, t, \xi)\mid &\leqslant K_{2}\big( |\xi|^{p-1} +a(x,t) |\xi|^{q-1}\big),
\end{aligned}
\end{equation}
where $K_{1}, K_{2}$ are positive constants and $p<q$.

Fix point $(x_{0}, t_{0}) \negthickspace\in\negthickspace \Omega_{T}$ and set $Q_{R_{1},R_{2}}(x_{0}, t_{0})\negthickspace:= \negthickspace Q^{-}_{R_{1},R_{2}}(x_{0}, t_{0}) \cup  Q^{+}_{R_{1},R_{2}}(x_{0}, t_{0}), Q^{-}_{R_{1},R_{2}}(x_{0}, t_{0})\negthickspace:= B_{R_{1}}(x_{0})\times (t_{0}-R_{2}, t_{0}), Q^{+}_{R_{1},R_{2}}(x_{0}, t_{0}):= B_{R_{1}}(x_{0})\times (t_{0}, t_{0}+R_{2}),  R_{1}, R_{2} >0.$

We assume that there exists positive continuous non-increasing function $\mu(r)\geqslant 1$
on the interval $(0,1) ,\quad  \lim\limits_{r \rightarrow 0}\mu(r) r^{1- \bar {b}} =0$ with some $\bar{b}\in (0,1)$ such that
\begin{equation}\label{eq1.3}
\begin{aligned}
 \mid a(x,t) \negthickspace-\negthickspace a(y,\tau) \mid \leq\negthickspace A \mu(r) r^{q-p},\,\, (x,t),(y,\tau) \in Q_{r,r}(x_{0},t_{0}) \subset \Omega_{T} ,
\end{aligned}
\end{equation}
with some $A>0$.
\begin{remark}\label{rem1.1}
Setting $\varPhi(x,t, v):= v^{p}+a(x,t) v^{q}$, $v>0$, we note (see e.g. \cite{SkrVoitPrepr2021}) that \eqref{eq1.3} yields the following $(\varPhi_{\lambda})$ and $(\varPhi_{\mu})$ conditions:
\begin{itemize}
\item [$(\varPhi_{\lambda})$]
there exists $\bar{K}>0$ depending only on $A$ such that for any $K>0$ there holds
$$\varPhi^{+}_{Q_{r,r}(x_{0},t_{0})}\bigg(\frac{v}{r}\bigg)\leqslant \bar{K}(1+K^{q-p}) \varPhi^{-}_{Q_{r,r}(x_{0},t_{0})}\bigg(\frac{v}{r}\bigg),\quad r<v\leqslant K\lambda(r), $$
\end{itemize}
where $\lambda(r)=[\mu(r)]^{-\frac{1}{q-p}},$
and
\begin{itemize}
\item[$(\varPhi_{\mu})$]
there exists $\bar{K}>0$ depending only on $A$ such that for any $K>0$ there holds
$$\varPhi^{+}_{Q_{r,r}(x_{0},t_{0})}\bigg(\frac{v}{r}\bigg)\leqslant \bar{K}(1+K^{q-p}) \mu(r) \varPhi^{-}_{Q_{r,r}(x_{0},t_{0})}\bigg(\frac{v}{r}\bigg),\quad r<v\leqslant K,$$
\end{itemize}
here  $\varPhi^{+}_{Q_{r,r}(x_{0},t_{0})}(v):=\max\limits_{(x,t)\in Q_{r,r}(x_{0},t_{0})}\varPhi(x,t, v)$,
\hspace*{8mm}$\varPhi^{-}_{Q_{r,r}(x_{0},t_{0})}(v):=\min\limits_{(x,t)\in Q_{r,r}(x_{0},t_{0})}\varPhi(x,t, v)$.
\end{remark}

In addition, we assume that the equation \eqref{eq1.1} is  degenerate at the point $(x_{0}, t_{0})$ which means that there exists $K_{3}$, $R_{0} >0$ such that the function
\begin{equation}\label{eq1.4}
\psi(x_{0}, t_{0} , v):= v^{p-2} + a(x_{0},t_{0}) v^{q-2} \,\, \text{is non-decreasing for}\,\, v\geqslant \dfrac{K_{3}}{R_{0}}.
\end{equation}
Particularly, this condition is valid if $p>2$ or $p\leqslant 2 < q$ and $a(x_{0}, t_{0}) >0$ (see \cite{SkrVoitNA20, SkrVoitPrepr2021}). In the case  $p=q>2$, these equations are classified as  degenerate because the diffusion term depends degenerately on the  gradient $\nabla u$.

Similarly, if we assume that $\psi(x_{0}, t_{0} , v)$ is non-increasing for $ v\geqslant \dfrac{K_{3}}{R_{0}}$ then equation \eqref{eq1.1} is singular at the point $(x_{0},t_{0})$.
 This condition is valid if $q<2$ or $p< 2 \leqslant q$ and $a(x_{0}, t_{0})=0$. This case will not be considered in this paper  we refer the reader to \cite{Skr} for the Harnack's inequality in the case $q<2$.

We\,\, will\,\, establish\,\, that\,\, non-negative \,bounded \,\,weak\,solutions \,of Eq.  \eqref{eq1.1} satisfy an intrinsic form of the Harnack's inequality in a neighborhood of $(x_{0}, t_{0})$. This property is basically characterized by the different types of degenerate behavior, according to the size of a coefficient $a(x, t)$ that determines the  phase. Indeed, on the set $\{a(x, t) = 0\}$ equation \eqref{eq1.1} has the growth of order $p$ with respect to the gradient $($this is so-called $p$-phase$)$, and at the same time this growth is of order $q$ if $a(x, t) > 0$ $($this corresponds to  $(p, q)$-phase$)$.

Before describing the main results, a few words concerning the history of the problem. The study of regularity of minima of functionals with non-standard growth has been initiated by Kolodij \cite{KolodijEll., KolodijPar.}, Zhikov \cite{ZhikIzv1983, ZhikIzv1986, ZhikJMathPh94, ZhikJMathPh9798, ZhikKozlOlein94}, Marcellini \cite{Marcellini1989, Marcellini1991} and Lieberman \cite{Lieberman91}, and in the last thirty years there has been growing interest and substantial development in the qualitative theory of second-order quasilinear elliptic and parabolic equations with so-called "log-conditions" (i.e. if $\mu(r)=1$). We refer the reader to the papers \cite{AlhutovKrash08,ZhikAlkhTSP11, AntZhikov2005, BarBog2014, BarColMing, BarColMingStPt16, BarColMingCalc.Var.18, BogDuzaar2012, BurchSkrPotAn, ColMing218, ColMing15, ColMingJFnctAn16, DienHarHastRuzVarEpn, DingZhangZhou2020, HarHastOrlicz, HarHastLeNuorNA2010, HarHasLee, HwangLieberman287, HwangLieberman288, Sur, WangNA2013, WinkZach2016, XuChen2006, Yao2014, Yao2015, ZhZhouXueNonAn2014,Yev,ShishYev,SkrVoit2021,BonSkr,ShSkr,LisSkr}  for the basic results, historical surveys and references.

The case when the condition \eqref{eq1.3} holds differs substantially from the logarithmic case. To our knowledge, there are a few results in this direction. Zhikov \cite{ZhikPOMI04} obtained a generalization of the logarithmic condition which guarantees the density of smooth functions in Sobolev space $W^{1, p(x)}(\Omega).$ Particularly, this result holds if $p(x) \geqslant p>1$ and
\begin{multline}\label{eq1.5}
\mid p(x)-p(y) \mid \leqslant  \frac{\log\mu(|x-y|)}{\mid \log \mid x-y \mid \mid},\quad x, y \in \Omega,~x \neq y, \\
\text{and}\quad \int\limits_{0}[\mu(r)]^{-\frac{n}{p}} \frac{dr}{r} = +\infty.
\end{multline}
 We note that the function $\mu(r)=\big[\log\dfrac{1}{r}\big]^{L},\quad 0\leqslant L \leqslant \dfrac{p}{n}$ satisfies the above condition.

 Interior continuity, continuity up to the boundary and Harnack`s inequality to the $p(x)-$Laplace equation were proved in \cite{AlhutovKrash08},  \cite{AlkhSurnAlgAn19} and  \cite{SurnPrepr2018} under the condition
\begin{equation}\label{eq1.6}
 \int\limits_{0}\,e^{-\gamma [\mu(r)]^{c}} \frac{dr}{r}=+\infty
\end{equation}
with some $\gamma$, $c>1$. Particularly, the function $\mu (r)=\big[\log \log \dfrac{1}{r}\big]^{L},$ $0<L<\dfrac{1}{c},$ satisfies the above condition.

These results were generalized in \cite{ShSkrVoit20, SkrVoitNA20} for a wide class of elliptic and parabolic equations with non-logarithmic Orlicz growth. Later, for elliptic and parabolic equations, the results from \cite{ShSkrVoit20, SkrVoitNA20} were substantially refined in \cite{HadzhySkrVoit, Skr, SkrVoitPrepr2021, SkrYev}. Interior continuity for double phase elliptic and parabolic equations  instead of condition \eqref{eq1.6} was proved under the condition
\begin{equation}\label{eq1.7}
\int\limits_{0} [\mu(r)]^{-\frac{1}{q-p}}\frac{dr}{r}=+\infty.
\end{equation}
In addition, in \cite{HadzhySkrVoit, Skr} Harnack's inequality was proved for quasilinear elliptic and singular ($q<2$) parabolic equations under the condition
\begin{equation}\label{eq1.8}
\int\limits_{0}  [\mu(r)]^{-\frac{1}{q-p} -\beta}\frac{dr}{r}=+\infty,
\end{equation}
with some $\beta >0 $. We note that this condition is worse than condition \eqref{eq1.7}, but at the same time it is  much better than condition \eqref{eq1.6}.

Harnack's inequality for non-uniformly elliptic conditions under  non-logarithmic condition was proved in \cite{HadSavSkrVoit}.
Later, continuity and Harnack's inequality under combining logarithmic, non-logarithmic, and non-uniformly elliptic conditions
were obtained in \cite{SavSkrYev}.

In this paper, we prove Harnack's inequality for nonnegative solutions to Eq. \eqref{eq1.1}  under the conditions \eqref{eq1.4} and
\eqref{eq1.8}.

To describe our results let us introduce the definition of a weak solution to Eq. \eqref{eq1.1}.

We say that $u$ is a bounded weak sub(super) solution to Eq. \eqref{eq1.1} if $u \in C_{\textrm{loc}}(0, T; L^{2}_{\textrm{loc}}(\Omega))\cap L_{\textrm{loc}}^{q}(0, T; W_{\textrm{loc}}^{1,q}(\Omega))\cap L^{\infty}(\Omega_{T}),$ and for any compact set $E \subset \Omega$ and any subinterval $[t_{1}, t_{2}]\subset (0, T]$ the integral identity
\begin{equation}\label{eq1.9}
\int\limits_{E}u \eta dx \bigg|^{t_{2}}_{t_{1}} + \int\limits^{t_{2}}\limits_{t_{1}}\int\limits_{E}\{-u\eta_{\tau}+ \mathbb{A}(x, \tau, \nabla u) \nabla \eta\} dx\, d\tau \leqslant (\geqslant) 0
\end{equation}
holds for any test function $\eta\negthickspace\geqslant \negthickspace0$, $\eta\negthickspace \in\negthickspace W^{1,2}(0, T; L^{2}(E))\cap L^{q}(0, T; W_{0}^{1,q}(E)).$

It would be technically convenient to have a formulation of a weak solution that involves $u_{t}.$ Let $\rho(x)\in C_{0}^{\infty}(\mathbb{R}^{n}),~\rho(x)\geqslant 0,~\rho(x)\equiv 0$ for $\mid x \mid > 1$ and $\int\limits_{\mathbb{R}^{n}}\rho (x)dx=1,$ and set

\noindent $\rho_{h}(x):= h^{-n}\rho(\frac{x}{h}),~u_{h}(x, t):= h^{-1}\int\limits_{t}\limits^{t+h}\int\limits_{\mathbb{R}^{n}}u(y, \tau)\rho_{h}(x-y)dy d\tau .$

Fix $t \in (0, T)$ and let $h>0$ be so small that $0<t<t+h<T.$ Now we take  $t_{1}=t,~~ t_{2}=t+h$ in \eqref{eq1.9}  and replace $\eta$ by $\int\limits_{\mathbb{R}^{n}}\eta(y, t)\rho_{h}(x-y)dy.$
Dividing by $h,$ since the test function does not depend on $\tau,$ we obtain
\begin{equation}\label{eq1.10}
\int\limits_{E\times \{t\}} \left(\frac{\partial u_{h}}{\partial t} \eta+[\mathbb{A}(x, t, \nabla u)]_{h} \nabla \eta\right)dx \leqslant(\geqslant)0,
\end{equation}
for all $t \in (0, T-h)$ and for all non-negative $\eta \in W^{1,q}_{0}(E).$

We refer to the parameters $M=\sup\limits_{\Omega_{T}} u, A, K_{1}, K_{2}, K_{3}, n,
p, q$  as our structural data, and we write $\gamma$ if it can be quantitatively determined a priori  in terms of the above quantities only. The generic constant $\gamma$ may change from line to line.

 As was already mentioned, the behavior of the solution in a neighborhood of a point $(x_{0},t_{0})$ depends on the value of the function $a(x_{0},t_{0})$. We will distinguish two cases: $a(x_{0},t_{0}) > 0$ (so-called $(p,q)$-phase)  and $a(x_{0},t_{0})=0$ (so-called $p$-phase).

First result is Harnack's inequality for positive solutions to \eqref{eq1.1} in the $(p,q)$-phase.

\begin{theorem}\label{th1.1}
Fix point  $(x_{0},t_{0}) \in \Omega_{T}$, let $u \in C(\Omega_{T})$ be a positive bounded weak solution to Eq. \eqref{eq1.1} and let the conditions \eqref{eq1.2}--\eqref{eq1.4} be fulfilled. Assume also that
$$
a(x_{0},t_{0}) >0.
$$
Then there exists $R > 0$, depending only on the data and $a(x_{0},t_{0})$, and there exist positive numbers $c, C$, depending only upon the data, such that for all $\rho \leqslant R^{2}$, either
\begin{equation}\label{eq1.11}
 u(x_{0}, t_{0}) \leqslant C \rho^{\frac{1}{2}},
\end{equation}
or
\begin{equation}\label{eq1.12}
u(x_{0}, t_{0})\leqslant C \inf\limits_{B_{\rho}(x_{0})} u (\cdot, t)
\end{equation}
 with $t \in (t_{0}+\frac{1}{2}\theta, t_{0}+\theta),~\theta:=\frac{\rho^{2}}{\psi\big(x_{0},t_{0},c\dfrac{u(x_{0},t_{0})}{\rho}\big)},$ provided that
$$Q_{\rho, \theta}(x_{0}, t_{0}) \subset Q_{\rho,\rho}(x_{0},t_{0}) \subset Q_{R,R^{2}}(x_{0}, t_{0}) \subset Q_{8R,(8R)^{2}}(x_{0},t_{0}) \subset \Omega_{T}.$$
The function $\psi(x_{0},t_{0},v)$, $v>0$ was defined in \eqref{eq1.4}.

\end{theorem}

\begin{remark}\label{rem1.2}
Choosing $R\negthickspace>\negthickspace0$ from the condition $A R^{q-p}\mu(\negmedspace R\negmedspace)\negthickspace=\negthickspace\frac{1}{4}a(\negthinspace x_{0},t_{0}\negthinspace)$,  we have $\frac{3}{4}a(x_{0},t_{0})\negthickspace\leqslant\negthickspace a(x,t)\negthickspace\leqslant \negthickspace \frac{5}{4}a(x_{0},t_{0})$ for any $(x,t)\in \negthickspace Q_{R,R^{2}}(x_{0},t_{0})$,
and by the Young inequality conditions \eqref{eq1.2} can be rewritten as follows:
$$
 \frac{1}{a(x_{0},t_{0})} \mathbb{A}(x, t, \xi)\xi  \geqslant \frac{K_{1}}{a(x_{0},t_{0})}\big(|\xi|^{p} +a(x,t)|\xi|^{q}\big)
\geqslant \frac{3}{4}K_{1}|\xi|^{q},\quad q>2,
$$
$$
 \frac{1}{a(x_{0},t_{0})}\mid \mathbb{A}(x, t, \xi)\mid  \leqslant \frac{K_{2}}{a(x_{0},t_{0})}\big(|\xi|^{p-1} +a(x,t)|\xi|^{q-1}\big)
 \leqslant $$
 $$
\leqslant \gamma(K_{2}) \left( \mid \xi \mid ^{q-1} + a(x_{0},t_{0})^{-\frac{q-1}{q-p}} \right).
$$
If \eqref{eq1.11} is violated we set $\tau=a(x_{0},t_{0}) t$ which transforms Eq. \eqref{eq1.1} into
$$
u_{\tau}-\textrm{div}\mathbb{\bar{A}}(x, \tau, \nabla u)=0 ,\quad \mathbb{\bar{A}}=\frac{1}{a(x_{0},t_{0})}\mathbb{A}
$$
in $Q_{\rho,\bar{\theta}}(x_{0},t_{0})~, \bar{\theta}=\gamma \dfrac{\rho^{q}}{u(x_{0},t_{0})^{q-2}}$. By the results of DiBenedetto, Gia-\\nazza and Vespri \cite{DiBGiVe1}, returning to the original coordinates, it follows that

$u(x_{0}, t_{0})\leqslant C \inf\limits_{B_{\rho}(x_{0})} u (\cdot, t),~t\in(t_{0}+\frac{1}{2}\bar{\theta}',t_{0}+\bar{\theta}')~,~
\bar{\theta}'= \dfrac{\bar{\theta}}{a(x_{0},t_{0})},$\\
provided that $u(x_{0},t_{0}) \geqslant \gamma \rho ~[a(x_{0},t_{0})]^{-\frac{1}{q-p}}$, which holds if \eqref{eq1.11} is violated. Indeed,
$$u(x_{0},t_{0}) \geqslant c \rho^{\frac{1}{2}} \geqslant c \rho R^{-1} \geqslant \gamma(A,M) \rho~ \mu^{\frac{1}{q-p}}(\bar{c}R)
 [a(x_{0},t_{0})]^{-\frac{1}{q-p}} \geqslant
  $$
  $$
  \geqslant\gamma(A,M) \rho ~[a(x_{0},t_{0})]^{-\frac{1}{q-p}}.$$
 To complete the proof of Theorem \ref{th1.1} we note that if inequality \eqref{eq1.11} is violated then
\begin{multline*}
 a(x_{0},t_{0})\bigg(\frac{u(x_{0},t_{0})}{\rho}\bigg)^{q-2} \leq \psi(x_{0},t_{0},\frac{u(x_{0},t_{0})}{\rho})\leqslant \\ \leqslant a(x_{0},t_{0})\bigg(\frac{u(x_{0},t_{0})}{\rho}\bigg)^{q-2}  \big\{1+ \frac{C^{p-q}}{a(x_{0},t_{0})}\rho^{\frac{q-p}{2}}\big\} \leqslant\\ \leqslant a(x_{0},t_{0})\bigg(\frac{u(x_{0},t_{0})}{\rho}\bigg)^{q-2}\big\{1+\frac{C^{p-q}}{a(x_{0},t_{0})} R^{q-p}\big\}\leqslant  \\ \leqslant \gamma(C,A) a(x_{0},t_{0})\bigg(\frac{u(x_{0},t_{0})}{\rho}\bigg)^{q-2} \big\{1+\frac{1}{\mu(R)}\big\}\leqslant\\
 \leqslant 2\gamma(C,A) a(x_{0},t_{0})\bigg(\frac{u(x_{0},t_{0})}{\rho}\bigg)^{q-2} .
\end{multline*}

\noindent Therefore, Theorem \ref{th1.1} is a consequence of the results by DiBenedetto, Gianazza and Vespri, we
refer the reader to \cite{DiBGiVe1} for the details.

\end{remark}

Our next result corresponds to the $p$-phase. Set
$$\lambda_{1}(r) := [\mu(r)]^{-\frac{1}{q-p}-n} .$$
 Further we will also suppose that with some $b_{1} \geqslant 1$  the following condition holds
\begin{equation}\label{eq1.13}
\lambda_{1}(\rho)\leqslant \bigg(\frac{\rho}{r}\bigg)^{b_{1}} \lambda_{1}(r) , \quad 0<r< \rho.
\end{equation}
Note that for the function $\mu(r)=[\log\frac{1}{r}]^{L}$, $L>0$  this condition is fulfilled automatically.
\begin{theorem}\label{th1.2}
Fix $(x_{0},t_{0}) \in \Omega_{T}$, let $u \in C(\Omega_{T})$ be a positive bounded weak solution to Eq. \eqref{eq1.1} and let  conditions \eqref{eq1.2}--\eqref{eq1.4}, \eqref{eq1.13} be fulfilled. Assume also that
$$
a(x_{0},t_{0}) =0,
$$
and
\begin{equation}\label{eq1.14}
(\mathbb{A}(x, t, \xi)-\mathbb{A}(x, t, \eta))(\xi-\eta)>0,~ \xi, \eta \in \mathbb{R}^{n},~ \xi\neq\eta.
\end{equation}
Then there exist positive numbers  $c$, $c_{1}$, $C$ depending only upon the data such that for all $\rho > 0$
either
\begin{equation}\label{eq1.15}
u(x_{0}, t_{0}) \leqslant C\frac{\rho}{\lambda_{1}(\rho)},
\end{equation}
or
\begin{equation}\label{eq1.16}
u(x_{0}, t_{0})\negthickspace\leqslant\negthickspace\frac{C}{\lambda_{1}(\rho)} \inf\limits_{B_{\rho}(x_{0})} u(\cdot, t),
\end{equation}
with $t\in (t_{0}+ c\theta,t_{0}+ c_{1}\theta), \theta:= \rho^{p} (\lambda_{1}(\rho)u(x_{0},t_{0}))^{2-p},$
provided that
$$Q_{\rho, \theta}(x_{0}, t_{0})\subset Q_{\rho,\rho}(x_{0},t_{0})\subset Q_{8\rho, 8\rho}(x_{0}, t_{0})  \subset \Omega_{T}.$$
\end{theorem}

\begin{remark}\label{rem1.3}
 We note that in the case  $\mu (\rho)=[\log\frac{1}{\rho}]^{L},\quad
0\leqslant L \leqslant \dfrac{q-p}{1+ n (q-p)},$
inequality \eqref{eq1.16} transformes into
\begin{equation}\label{eq1.17}
u(x_{0}, t_{0})\leqslant C \log \frac{1}{\rho} \inf\limits_{B_{\rho}(x_{0})} u (\cdot, t_{0}+\theta),\quad \theta= c\rho^{p}\bigg(\frac{\log\frac{1}{\rho}}{u(x_{0},t_{0})}\bigg)^{p-2}.
\end{equation}

\end{remark}

We would like to mention the approach taken in this paper. To prove our results we use DiBenedetto's approach \cite{DiBenedettoDegParEq}, who developed  innovative intrinsic scaling methods for degenerate and singular parabolic equations. For the p-Laplace evolution equation the intrinsic Harnack's inequality was proved in the  papers \cite{DiBGiVe1, DiBGiVes2}.

The difficulties arising in the proof of our Theorem \ref{th1.2} are related to the so-called theorem on the expansion of positivity. Roughly speaking, having information on the measure of the "positivity set" of $u$ over the ball $B_{r}(\bar{x})$ for some time level $\bar{t}$:
$$\mid \{ x \in B_{r}(\bar{x}): u(x, \bar{t}) \geqslant N \}\mid \geqslant \alpha(r) \mid B_{r} (\bar{x}) \mid,$$
\noindent with some $r>0,~ N>0$ and $\alpha(r)\in (0, 1), ~ \alpha(r)\rightarrow 0,$ as $r\rightarrow 0,$ and using the standard DiBenedetto's arguments, we inevitably arrive at the estimate
$$u(x, t) \geqslant \frac{N}{\gamma_{1}} \exp\big(-\gamma_{1}[\alpha(r)\mu(r)]^{-\gamma_{2}}\big) , ~ x \in B_{2r}(\bar{x}),$$
\noindent for some time level $t>\bar{t}$ and with some $\gamma_{1}, \gamma_{2}>1.$ This estimate leads us to a condition similar to that of \eqref{eq1.6} (see, e.g. \cite{ShSkrVoit20, SkrVoitPrepr2021}). To avoid this, we use a workaround that goes back to Maz'ya  \cite{Maz} and Landis \cite {Landis1, Landis2} papers. So, in Section $3$ we use the auxiliary solutions and prove integral and pointwise estimates of these solutions.

Another difficulty arising in the proof of Theorem \ref{th1.2} is also closely related to the theorem on the expansion of positivity. Namely,
if we expand the positivity  from the small ball $B_{r}(\bar{x})$ and time level $\bar{t}$ to the large ball $B_{\rho}(x_{0})$ and some time level $t >\bar{t}$ in the case when $a(x_{0},t_{0})=0$ and
$\max\limits_{Q_{4r, 4r}(\bar{x},\bar{t})}a(x, t)\geqslant 4A~\mu(4r)(4r)^{q-p}$ for some $(\bar{x},\bar{t}) \in Q_{\rho,\rho}(x_{0}, t_{0})$, we need to obtain the lower bound of a solution independent of $\max\limits_{Q_{4r, 4r}(\bar{x},\bar{t})} a(x,t)$. For this, we also use the auxiliary solutions defined in  Section \ref{Sect3}.

The rest of the paper contains a proof of the above theorems. In Section \ref{Sect2} we collect some auxiliary propositions. Section \ref{Sect3} contains the proof of the required integral and pointwise estimates of auxiliary solutions.\,\, Expansion\,\, of \,\,positivity\,\, is\,\, proved \,\,in\,\, Section\,\, \ref{Sect4}. \,\,In \,\,Sec-\\tion \ref{Sect5}~we give a proof of Harnack's inequality using pointwise estimates of auxiliary solutions.

\section{Auxiliary material and integral estimates of solutions}\label{Sect2}

\textbf{2.1. An auxiliary proposition}

The following  lemma will be used in the sequel, it is  the well-known De Giorgi-Poincare lemma (see \cite{DiBenedettoDegParEq}, Chapter I).

\begin{lemma}\label{lem2.1}
{\it Let $u \in W^{1,1}(B_{r}(y))$ for some $r > 0$, and $y \in \mathbb{R}^{n}$. Let $k, l$ be real
numbers such that $k < l$. Then there exists a constant $\gamma$ depending only on $n$ such that
\begin{equation*}
(l-k) |A_{k,r}||B_{r}(y)\setminus A_{l,r}| \leqslant \gamma r^{n+1} \int\limits_{A_{l,r}\setminus A_{k,r}} |\nabla u| dx,
\end{equation*}
where $A_{k,r} = B_{r}(y)\cap \{u < k\}$.
}
\end{lemma}

\textbf{2.2. Local material and energy estimates}

Further we will need the following local energy estimate.

\begin{lemma}\label{lem2.2}
Let u be a bounded weak solution to \eqref{eq1.1} in $\Omega_{T}$. Then for any cylinder $Q^{-}_{r,\theta}(\bar{x},\bar{t})\subset \Omega_{T}$, any $k \in \mathbb{R}^{1}$, any $\sigma \in(0,1)$ and any smooth $\zeta(x,t)$ which vanishes on $\partial B_{r}(\bar{x})\times(\bar{t}-\theta,\bar{t})$ and $|\nabla \zeta| \leqslant \dfrac{1}{\sigma r}$ one has
\begin{multline}\label{eq2.1}
\sup\limits_{\bar{t}-\theta<t<\bar{t}}\int\limits_{B_{r}(\bar{x})}(u-k)^{2}_{\pm} \zeta^{q}dx +
\gamma^{-1}\iint\limits_{Q^{-}_{r,\theta}(\bar{x},\bar{t})}\varPhi(x,t,|\nabla (u-k)_{\pm}|)\zeta^{q} dxdt \leqslant\\
\leqslant \int\limits_{B_{r}(\bar{x})} (u-k)^{2}_{\pm}\zeta^{q}(x,\bar{t}-\theta) dx+
\gamma\iint\limits_{Q^{-}_{r,\theta}(\bar{x},\bar{t})} (u-k)^{2}_{\pm}|\zeta_{t}|\zeta^{q-1}dxdt+\\
+\frac{\gamma}{\sigma^{q}}\varPhi^{+}_{Q_{r,\theta}(\bar{x}, \bar{t})}\bigg(\frac{M_{\pm}(k,r,\theta)}{r}\bigg)\big|Q_{r,\theta}(\bar{x}, \bar{t})\cap \big\{(u-k)_{\pm}>0\big\}\big|,
 \end{multline}
here $M_{\pm}(k,r,\theta):=\sup\limits_{Q_{r,\theta}(\bar{x},\bar{t})} (u-k)_{\pm}$.
\end{lemma}

\begin{proof}Test identity \eqref{eq1.10} by $\eta=(u_{h}-k)_{\pm}\zeta^{q}$, integrating it over $(\bar{t}-\theta,t),t\in (\bar{t}-\theta,\bar{t})$
and then integrating by parts in the term containing $\dfrac{\partial u_{h}}{\partial t}$. Letting $h \rightarrow 0$, using conditions \eqref{eq1.2} and the Young inequality, we arrive at  the required inequality \eqref{eq2.1}, which completes the proof of the lemma.
\end{proof}

The following lemma will be used in the sequel.
\begin{lemma}\label{lem2.3}
 Let u be a bounded non-negative weak solution to Eq. \eqref{eq1.1} in $\Omega_{T}$.  Suppose  that for some $ Q^{-}_{4r, 4r}(\bar{x}, \bar{t})\subset \Omega_{T}$
\begin{equation}\label{eq2.2}
|\{B_{r}(\bar{x}) : u(\cdot,\bar{t})\leqslant  N\}| \leqslant (1-\alpha_{0}) |B_{r}(\bar{x})|,
\end{equation}
for some $0<N<M$ and some $\alpha_{0} \in (0,1)$. Then there exist numbers $\varepsilon_{0}$, $\delta_{0}$ depending only on the known data and $\alpha_{0}$ such that for all $t \in (\bar{t}, \bar{t} +\bar{\theta})$
\begin{equation}\label{eq2.3}
|\{B_{r}(\bar{x}) : u(\cdot,t) \leqslant  \varepsilon_{0} N \}| \leqslant \left(1- \frac{\alpha^{2}_{0}}{2}\right) |B_{r}(\bar{x})|,
\end{equation}
\begin{equation}\label{eq2.4}
\bar{\theta} =\frac{\delta_{0} r^{2}}{\psi^{+}_{Q_{4r, 4r}(\bar{x},\bar{t})}(\frac{N}{r})},\quad \psi^{+}_{Q_{4r, 4r}(\bar{x},\bar{t})}(v):= \frac{1}{v^{2}} \varPhi^{+}_{Q_{4r, 4r}(\bar{x},\bar{t})}(v),
\end{equation}
provided that $\bar{\theta} \leqslant 4 r$.

\end{lemma}

\begin{proof}We use Lemma \ref{lem2.2} in the cylinder $Q^{+}_{r,\bar{\theta}}(\bar{x},\bar{t})$ and $\zeta \in C^{\infty}_{0}(B_{r}(\bar{x}))$,
 $0\leqslant \zeta \leqslant 1$,  $\zeta(x) =1$ in $B_{(1-\sigma)r}(\bar{x})$,  $|\nabla \zeta|\leqslant \dfrac{1}{\sigma r}$, where $\sigma \in (0,1)$ will be fixed later. By Lemma \ref{lem2.2} and \eqref{eq2.2} it follows that
\begin{multline*}
\int\limits_{B_{(1-\sigma)r}(\overline{x})\times\{t\}}(N-u)_{+}^2dx \leqslant  N^{2}\big|\big\{B_{r}(x_{0}) :u(\cdot,\bar{t})\leqslant N \big\}\big|+\\+\gamma \sigma^{-q}\big\{r^{-p} N^{p}+\max\limits_{Q_{4r, 4r}(\bar{x},\bar{t})}a(x,t)r^{-q} N^{q}\big\} \bar{\theta}\big|B_{r}(\bar{x})\big|\leqslant\\
 \leqslant N^{2}\big\{1-\alpha_{0}+\gamma \sigma^{-q}\delta_{0}\big\}\big|B_{r}(\bar{x})\big|.
\end{multline*}
We infer from this that for all $t\in (\bar{t},\bar{t}+\bar{\theta})$

$$\big|\big\{B_{r}(\bar{x}) : u(\cdot,t) \leqslant \varepsilon_{0} N \big\}\big| \leqslant\bigg(n\sigma+\frac{1-\alpha_{0}}{(1-\varepsilon_{0})^{2}}+\frac{\gamma\sigma^{-q}\delta_{0}}{(1-\varepsilon_{0})^{2}}\bigg) \big|B_{r}(\bar{x})\big| .$$
Choosing $\sigma$ such that $n \sigma\leqslant \frac{1}{4}\alpha_{0}^{2}$, and $\varepsilon_{0}$ such that $ \dfrac{1}{(1-\varepsilon_{0})^{2}}\leqslant 1+\alpha_{0}$, and finally, choosing $\delta_{0}$ such that $\delta_{0} \gamma \sigma^{-q} (1+\alpha_{0})\leqslant \frac{1}{4}\alpha_{0}^{2}$, we arrive at the required \eqref{eq2.3}, which completes the proof of the lemma.
\end{proof}

\textbf{2.3. De Giorgi type lemmas}

The next lemmas will be used in the sequel and they are a consequence of the Sobolev embedding theorem and Lemma \ref{lem2.2}.

\begin{lemma}\label{lem2.4}
 Let u be a bounded non-negative weak solution to Eq. \eqref{eq1.1} in $\Omega_{T}$. Let $(\bar{x},\bar{t})$ be some point in $\Omega_{T}$
such that  $Q_{r, \theta}(\bar{x},\bar{t}) \subset Q_{4r, 4r}(\bar{x}, \bar{t}) \subset \Omega_{T}$. Fix $\xi_{0} \in (0,1)$ and $N\in (0,M)$, then there exists number  $\nu \in (0,1)$ depending only on the data and $\xi_{0}$, $r$, $\theta$, $N$ such that if
\begin{equation}\label{eq2.5}
|\{ Q^{-}_{r,\theta}(\bar{x},\bar{t}) : u \leqslant N \}| \leqslant \nu [\mu(4r)]^{-n} |Q^{-}_{r,\theta}(\bar{x},\bar{t})|,
\end{equation}
then
\begin{equation}\label{eq2.6}
u(x,t) \geqslant \xi_{0} N,\quad \text{for a.a.}\quad (x,t) \in Q^{-}_{\frac{r}{2},\frac{\theta}{2}}(\bar{x},\bar{t}).
\end{equation}

Likewise, assume that  with some $\gamma_{0} >0$
\begin{equation}\label{eq2.7}
\bigg(\frac{N}{r}\bigg)^{q-p}\max\limits_{Q_{4r, 4r}(\bar{x}, \bar{t})}a(x, t) \leqslant \gamma_{0},
\end{equation}
then there exists number $\nu\in (0,1)$ depending only on the data and $\xi_{0}$, $r$, $\theta$, $N$, $\gamma_{0}$ such that if
\begin{equation}\label{eq2.8}
|\{ Q^{-}_{r,\theta}(\bar{x},\bar{t}) : u \leqslant N  \}| \leqslant \nu  |Q^{-}_{r,\theta}(\bar{x},\bar{t})|,
\end{equation}
then
\begin{equation}\label{eq2.9}
u(x,t) \geqslant \xi_{0} N ,\quad \text{for a.a.} (x,t) \in Q^{-}_{\frac{r}{2},\frac{\theta}{2}}(\bar{x},\bar{t}).
\end{equation}

\end{lemma}

\begin{proof}For $j=0,1,2, \ldots,$ we define the sequences $r_{j}:=\dfrac{r}{2}(1+2^{-j}), \,\, \theta_{j}:=\dfrac{\theta}{2}(1+2^{-j}), \,\,
\overline{r}_{j}:=\dfrac{r_{j}+r_{j+1}}{2}, \quad
\overline{\theta}_{j}:=\dfrac{\theta_{j}+\theta_{j+1}}{2},
B_{j}:=B_{r_{j}}(x_{0}), \quad \overline{B}_{j}:=B_{\overline{r}_{j}}(x_{0}),
\quad Q_{j}:=Q^{-}_{r_{j}, \theta_{j}}(x_{0},\overline{t}), \,\,
\overline{Q}_{j}:=Q^{-}_{\overline{r}_{j}, \overline{\theta}_{j}}(x_{0},\overline{t}),\quad
k_{j}:=\xi_{0} N + (1-\xi_{0}) \frac{N}{2^{j}},
\quad A_{j,k_{j}}:= Q_{j}\cap \{u<k_{j}\}, \quad
\overline{A}_{j,k_{j}}:= \overline{Q}_{j}\cap \{u<k_{j}\}.$
Let $\zeta_{j}\in C_{0}^{\infty}(\overline{B}_{j})$, $0\leqslant\zeta_{j}\leqslant1$,
$\zeta_{j}=1$ in $B_{j+1}$ and $|\nabla \zeta_{j}|\leqslant \gamma2^{j}/r$.
Consider also the function $\chi_{j}(t)=1$ for $t\geqslant \overline{t}-\theta_{j+1}$,
$\chi_{j}(t)=0$ for $t< \overline{t}-\theta_{j}$, $0\leqslant\chi_{j}(t)\leqslant1$
and $|\chi'_{j}|\leqslant \gamma2^{j}/\theta$.

Lemma \ref{lem2.2} with such choices implies that
\begin{multline}\label{eq2.10}
\sup\limits_{\bar{t}-\theta_{j}<t<\bar{t}}\int\limits_{B_{j}} (u-k_{j})^{2}_{-}\zeta^{q}_{j}\chi^{q}_{j} dx
+\iint\limits_{Q_{j}} \varPhi(x,t,|\nabla(u-k_{j})_{-}|)\zeta^{q}_{j}\chi^{q}_{j} dx dt \leqslant\\
\leqslant \gamma 2^{j\gamma}\bigg(\theta^{-1} k^{2}_{j} +r^{-2}k_{j}^{2}\psi^{+}_{Q_{4r,4r}(\bar{x},\bar{t})}\bigg(\frac{k_{j}}{r}\bigg)\bigg)|A_{j,k_{j}}|\leqslant \\ \leqslant \gamma 2^{j\gamma}\varPhi^{+}_{Q_{4r,4r}(\bar{x},\bar{t})}\bigg(\frac{N}{r}\bigg)\bigg(1+\frac{r^{2}}
{\theta \psi^{+}_{Q_{4r,4r}(\bar{x},\bar{t})}\big(\frac{N}{r}\big)}\bigg)|A_{j,k_{j}}|,
\end{multline}
where $\psi^{+}_{Q_{4r,4r}(\bar{x},\bar{t})}(\frac{N}{r})$ was defined in \eqref{eq2.4}.

By the Young inequality and \eqref{eq2.10} we have
$$
\iint\limits_{Q_{j}} |\nabla \varPhi^{-}_{Q_{4r,4r}(\bar{x},\bar{t})}\bigg(\frac{(u-k_{j})_{-}}{r}\bigg)|\zeta_{j}^{q}\chi_{j}^{q} dxdt \leqslant
 $$
 $$
 \leqslant  \frac{\gamma}{r} \iint\limits_{Q_{j}} \varphi^{-}_{Q_{4r,4r}(\bar{x},\bar{t})}\bigg(\frac{(u-k_{j})_{-}}{r}\bigg) |\nabla(u-k_{j})_{-}| \zeta_{j}^{q} \chi_{j}^{q} dxdt\leqslant
 $$
\begin{multline}\label{eq2.11}
\leqslant  \frac{\gamma}{r}\iint\limits_{Q_{j}} \varphi\bigg(x,t,\frac{(u-k_{j})_{-}}{r}\bigg) |\nabla(u-k_{j})_{-}| \zeta_{j}^{q} \chi_{j}^{q} dxdt\leqslant\\ \leqslant  \frac{\gamma}{r}\iint\limits_{Q_{j}}\varPhi\bigg(x,t,\frac{(u-k_{j})_{-}}{r}\bigg)\zeta_{j}^{q}\chi_{j}^{q}dxdt +  \frac{\gamma}{r}\iint\limits_{Q_{j}} \varPhi(x,t,|\nabla(u-k_{j})_{-}|)\zeta^{q}_{j}\times\\
\times\negthickspace\chi^{q}_{j} dx dt  \negthickspace\leqslant \negthickspace \gamma
\frac{2^{j\gamma}}{r}\varPhi^{+}_{Q_{4r,4r}(\bar{x},\bar{t})}\bigg(\frac{N}{r}\bigg)\left(1+\frac{r^{2}}
{\theta \psi^{+}_{Q_{4r,4r}(\bar{x},\bar{t})}\big(\frac{N}{r}\big)}\right)|A_{j,k_{j}}|.
\end{multline}
By \eqref{eq2.10}, \eqref{eq2.11}, using the Sobolev embedding theorem and H\"{o}lder's inequality, we obtain
\begin{multline}\label{eq2.12}
\bigg(\frac{(1-\xi_{0})N}{2^{j+1}}\bigg)^{\frac{2}{n}}\varPhi^{-}_{Q_{4r,4r}(\bar{x},\bar{t})}\bigg(\frac{(1-\xi_{0})N}{2^{j+1}r}\bigg) |A_{j+1,k_{j+1}}| \leqslant \\
\leqslant\iint\limits_{Q_{j}}(u-k_{j})_{-}^{\frac{2}{n}}\varPhi^{-}_{Q_{4r,4r}(\bar{x},\bar{t})}
\bigg(\frac{(u-k_{j})_{-}}{r}\bigg)(\zeta_{j}\chi_{j})^{1+\frac{1}{n}} dxdt \leqslant \\
\leqslant \gamma \bigg(\sup\limits_{\bar{t}-\theta_{j}<t<\bar{t}}\int\limits_{B_{j}} (u-k_{j})^{2}_{-}\zeta^{q}_{j}\chi^{q}_{j} dx\bigg)^{\frac{1}{n}}\times\\
\times\iint\limits_{Q_{j}} \left|\nabla \left(\varPhi^{-}_{Q_{4r,4r}(\bar{x},\bar{t})}\bigg(\frac{(u-k_{j})_{-}}{r}\bigg)\zeta_{j}^{q}\chi_{j}^{q}\right)\right|dxdt\leqslant\\
\leqslant\negthickspace\gamma \frac{2^{j\gamma}}{r}\bigg[\varPhi^{+}_{Q_{4r,4r}\negthickspace(\bar{x},\bar{t})}\bigg(\frac{N}{r}\bigg)\bigg]^{1+
\frac{1}{n}}\negthickspace\negthickspace\left(\negthickspace1\negthickspace+\negthickspace\frac{r^{2}}{\theta \psi^{+}_{Q_{4r,4r}(\bar{x},\bar{t})}\big(\frac{N}{r}\big)}\right)^{1+\frac{1}{n}}\negthickspace\negthickspace|A_{j,k_{j}}|^{1+\frac{1}{n}},
\end{multline}
which by $(\varPhi_{\mu})$ condition yields
\begin{equation*}
y_{j+1}:=\frac{|A_{j+1,k_{j+1}}|}{|Q_{j+1}|} \leqslant \gamma 2^{j\gamma} \mu(4r)(1-\xi_{0})^{-q-\frac{2}{n}}\bigg[\psi^{+}_{Q_{4r,4r}(\bar{x},\bar{t})}\bigg(\frac{N}{r}\bigg)\frac{\theta}{r^{2}}\bigg]^{\frac{1}{n}}\times
\end{equation*}
\begin{equation*}
\times\bigg(1+\frac{r^{2}}{\theta \psi^{+}_{Q_{4r,4r}(\bar{x},\bar{t})}\big(\frac{N}{r}\big)}\bigg)^{1+\frac{1}{n}}y_{j}^{1+\frac{1}{n}}.
\end{equation*}
From this, by iteration, it follows that $\lim\limits_{j \rightarrow +\infty} |A_{j,k_{j}}|=0$, provided that $\nu$ is chosen to satisfy
\begin{equation}\label{eq2.13}
\nu= \gamma^{-1} (1-\xi_{0})^{nq+2}\frac{r^{2}}{\theta\psi^{+}_{Q_{4r,4r}(\bar{x},\bar{t})}\big(\frac{N}{r}\big) } \bigg(1+\frac{r^{2}}{\theta \psi^{+}_{Q_{4r,4r}(\bar{x},\bar{t})}\big(\frac{N}{r}\big)}\bigg)^{-n-1},
\end{equation}
which proves \eqref{eq2.6}.

To prove \eqref{eq2.9} we choose $k_{j}:=\xi_{0} N  + (1-\xi_{0}) \dfrac{N}{2^{j}}$, by condition \eqref{eq2.7}
\begin{equation*}
\frac{1}{r^{q}}\max\limits_{Q_{4r, 4r}(\bar{x},\bar{t})}a(x,t) (u-k_{j})_{-}^{q} \leqslant \frac{\gamma_{0}}{r^{p}}(u-k_{j})_{-}^{p}
\end{equation*}
and
\begin{equation*}
\bigg(\frac{N}{r}\bigg)^{p-2}\leqslant \psi^{+}_{Q_{r,r}(\bar{x},\bar{t})}\bigg(\frac{N}{r}\bigg)\leqslant (1+\gamma_{0})\bigg(\frac{N}{r}\bigg)^{p-2}.
\end{equation*}

Therefore inequalities \eqref{eq2.10}-\eqref{eq2.12} can be rewritten as follows:

\begin{multline*}
\sup\limits_{\bar{t}-\theta_{j}<t<\bar{t}}\int\limits_{B_{j}} (u-k_{j})^{2}_{-}\zeta^{q}_{j}\chi^{q}_{j} dx
+\iint\limits_{Q_{j}} |\nabla(u-k_{j})_{-}|^{p}\zeta^{q}_{j}\chi^{q}_{j} dx dt \leqslant\\
\leqslant \gamma 2^{j\gamma}\bigg(\frac{N}{r}\bigg)^{p}\bigg(1+\frac{r^{p}}{\theta N^{p-2}}\bigg)|A_{j,k_{j}}|,
\end{multline*}

and
\begin{multline*}
\bigg(\frac{(1-\xi_{0})N}{2^{j+1}}\bigg)^{p+\frac{2}{n}} |A_{j+1,k_{j+1}}| \leqslant \\ \leqslant \gamma 2^{j\gamma} r^{p-1}
\bigg(\frac{N}{r}\bigg)^{p(1+\frac{1}{n})}\bigg(1+\frac{r^{p}}{\theta N^{p-2}}\bigg)^{1+\frac{1}{n}}|A_{j,k_{j}}|^{1+\frac{1}{n}},
\end{multline*}
from which it follows that
\begin{equation*}
y_{j+1}\negthickspace:=\frac{|A_{j+1,k_{j+1}}|}{|Q_{j+1}|} \negthickspace\leqslant\negthickspace \gamma 2^{j\gamma}(1-\xi_{0})^{-p-\frac{2}{n}}\bigg(\frac{\theta N^{p-2}}{r^{p}}\bigg)^{\frac{1}{n}}\bigg(1+\frac{r^{p}}{\theta N^{p-2}}\bigg)^{1+\frac{1}{n}} y_{j}^{1+\frac{1}{n}},
\end{equation*}
which yields $\lim\limits_{j \rightarrow +\infty} |A_{j,k_{j}}|=0$, provided that $\nu$ is chosen to satisfy
\begin{equation}\label{eq2.14}
\nu= \gamma^{-1} (1-\xi_{0})^{np+2}\frac{r^{p}}{\theta N^{p-2}} \bigg(1+\frac{r^{p}}{\theta N^{p-2} }\bigg)^{-n-1},
\end{equation}
which proves \eqref{eq2.9}. This completes the proof of the lemma.

\end{proof}

\section{Integral and pointwise estimates of auxiliary solutions}\label{Sect3}

Fix $(x_{0}, t_{0})\in \Omega_{T}$ such that $a(x_{0},t_{0})=0$
and let $(\bar{x},\bar{t}) \in Q_{\rho,\rho}(x_{0},t_{0})\subset Q_{8\rho,8\rho}(x_{0},t_{0}) \subset \Omega_{T}$. Let $0<r\leqslant \frac{1}{2}\rho$ , $E\subset B_{r}(\bar{x})$,  $\mid E \mid>0$,  $0<N\leqslant M$, and  we also suppose that
\begin{equation}\label{eq3.1}
N \lambda(r) \frac{|E|}{\rho^{n}}\geqslant \rho.
\end{equation}
We will consider separately two cases: $\max\limits_{Q_{8r,8r}(\bar{x},\bar{t})}a(x, t)\leqslant 4A\mu(8r)(8r)^{q-p}$
and  $\max\limits_{Q_{8r, 8r}(\bar{x}, \bar{t})}a(x, t)\geqslant 4A~\mu(8r)(8r)^{q-p}.$ In the case $\max\limits_{Q_{8r,8r}(\bar{x}, \bar{t})}a(x, t)\leqslant 4A~\mu(8r)(8r)^{q-p}$, we consider the function $v(x,t)=v_{r,N}(x,t,\bar{x},\bar{t}) \in C(\bar{t}, \bar{t}+8\tau_{1}; L^{2}(B_{8\rho}(\bar{x}))) \cap L^{q}(\bar{t}, \bar{t}+8\tau_{1}; W_{0}^{1, q}(B_{8\rho}(\bar{x})))$ with $\tau_{1}=\rho^{p}\times\\\times \left(N\lambda(r)\dfrac{|E|}{\rho^{n}}\right)^{2-p}$ as the solution of the following problem

\begin{equation}\label{eq3.2}
 v_{t}-div \mathbb{A}(x, t, \nabla v)=0,\quad(x, t)\in Q_{1}:= B_{8\rho}(\bar{x})\times(\bar{t}, \bar{t}+8\tau_{1}),
\end{equation}

\begin{equation}\label{eq3.3}
 v(x,t)=0,\quad(x,t)\in \partial B_{8\rho}(\bar{x})\times(\bar{t},\bar{t}+8\tau_{1}),
\end{equation}

\begin{equation}\label{eq3.4}
v(x, \bar{t})=N \lambda(r) \chi(E),\quad x\in B_{8\rho}(\bar{x}).
\end{equation}

In addition, the integral identity

\begin{equation}\label{eq3.5}
\int\limits_{B_{8\rho}(\bar{x})\times\{t\}} \left(\frac{\partial v_{h}}{\partial t}\eta + [\mathbb{A}(x, t, \nabla v)]_{h} \nabla \eta\right) dx = 0,
\end{equation}
holds for all $t\in (\bar{t}, \bar{t}+8\tau_{1}-h)$ and for all $\eta\in W^{1, q}_{0}(B_{8\rho}(\bar{x})).$ Here $v_{h}$ is defined similarly to \eqref{eq1.10}.

In the case $\max\limits_{Q_{8r,8r}(\bar{x}, \bar{t})}a(x, t)\geqslant 4A~\mu(8r)(8r)^{q-p}$, by our assumptions there exists $\bar{\rho} \in (0,\rho)$, such that
$$\max\limits_{Q_{4\bar{\rho}, 4\bar{\rho}}(\bar{x},\bar{t})}a(x,t) \geqslant 4A \mu(4\bar{\rho})(4\bar{\rho})^{q-p},\,\,
\max\limits_{Q_{8\bar{\rho}, 8\bar{\rho}}(\bar{x},\bar{t})}a(x,t) \leqslant 4A \mu(8\bar{\rho})(8\bar{\rho})^{q-p}.$$
Let $\rho_{0}$ be the maximal number satisfying the above conditions. We consider the function $w(x,t)=w_{r,N}(x,t,\bar{x},\bar{t}) \in C(\bar{t}, \bar{t}+8\tau_{2}; L^{2}(B_{8\rho_{0}}(\bar{x})))\cap \\ \cap L^{q}(\bar{t}, \bar{t}+8\tau_{2}; W_{0}^{1, q}(B_{8\rho_{0}}(\bar{x}))) , \tau_{2}=\rho_{0}^{p}\left(N\lambda(r)\dfrac{|E|}{\rho_{0}^{n}}\right)^{2-p}$ as the solution of the following problem
\begin{equation}\label{eq3.6}
 w_{t}-div \mathbb{A}(x, t, \nabla w)=0,\quad(x, t)\in Q_{2}:= B_{8\rho_{0}}(\bar{x})\times(\bar{t}, \bar{t}+8\tau_{2}),
\end{equation}
\begin{equation}\label{eq3.7}
 w(x,t)=0,\quad(x,t)\in \partial B_{8\rho_{0}}(\bar{x})\times(\bar{t},\bar{t}+8\tau_{2}),
\end{equation}

\begin{equation}\label{eq3.8}
w(x, \bar{t})=N \lambda(r) \chi(E),\quad x\in B_{8\rho_{0}}(\bar{x}).
\end{equation}

In addition, the integral identity

\begin{equation}\label{eq3.9}
\int\limits_{B_{8\rho_{0}}(\bar{x})\times\{t\}} \left(\frac{\partial w_{h}}{\partial t}\eta + [\mathbb{A}(x, t, \nabla w)]_{h} \nabla \eta\right) dx = 0,
\end{equation}
holds for all $t\in (\bar{t}, \bar{t}+8\tau_{2}-h)$ and for all $\eta\in W^{1, q}_{0}(B_{8\rho_{0}}(\bar{x})).$ Here $w_{h}$ is defined similarly to \eqref{eq1.10}. The existence of the solutions $v$ and $w$ follows from the general theory of monotone operators. Testing
\eqref{eq3.5} by $\eta=(v_{h})_{-}$ and $\eta=(v_{h}-N)_{+},$ integrating it over $(\bar{t}, t),~ t\in (\bar{t}, \bar{t}+8\tau_{1})$ and letting $h\rightarrow 0$, we obtain that $0\leqslant v \leqslant N \leqslant \lambda(r)M.$ Similarly we obtain that $0\leqslant w \leqslant N \leqslant \lambda(r)M.$

Set $\mathcal{D}(\rho):= \left\{(x, t): \mid x-\bar{x} \mid^{p}+(t-\bar{t}) \bigg(N\lambda(r)\dfrac{ \mid E \mid}{\rho^{n}}\bigg)^{p-2}\leqslant \rho^{p}\right\}.$

\begin{lemma}\label{lem3.1}
{\it Next inequalities hold
\begin{equation}\label{eq3.10}
v(x, t) \leqslant \gamma N \lambda(r) \frac{\mid E \mid}{\rho^{n}},\quad (x, t)\in Q_{1}\setminus \mathcal{D}(\rho),
\end{equation}
\begin{equation}\label{eq3.11}
w(x,t) \leqslant  \gamma N \lambda(r) \frac{\mid E \mid}{\rho_{0}^{n}},\quad  (x, t)\in Q_{2}\setminus \mathcal{D}(\rho_{0}).
\end{equation}
}
\end{lemma}

\begin{proof} For fixed $\sigma\in (0, 1),~ \rho \leqslant s \leqslant s(1+\sigma)\leqslant 2\rho,$ and $j=0, 1, 2, ...$ set $s_{j}:= s(1+\sigma)- \dfrac{\sigma s}{2^{j}},~ k_{j}:= k-2^{-j}k,$ $k>0,$ $\mathcal{D}_{j}:= \big\{(x, t):\mid x-\bar{x} \mid^{p}+(t-\bar{t})\bigg( N\lambda(r) \dfrac{ \mid E \mid}{\rho^{n}}\bigg)^{p-2}\leqslant s_{j}^{p}\big\},$ and let $M_{0}:= \sup\limits_{Q_{1}\setminus\mathcal{D}_{0}}v,~ M_{\sigma}:= \sup\limits_{Q_{1}\setminus\mathcal{D}_{\infty}}v,$ and consider the function $\zeta \in C^{\infty}(\mathbb{R}^{n+1}),~ 0\leqslant\zeta\leqslant 1,$ $\zeta = 0$ in $\mathcal{D}_{j}, \zeta=1$ in $Q_{1}\setminus\mathcal{D}_{j+1},~ \mid \nabla \zeta \mid\leqslant \dfrac{2^{j+1}}{\sigma s},$ $\mid \zeta_{t} \mid \leqslant 2^{p(j+1)}(\sigma s)^{-p}\bigg( N\lambda(r)\dfrac{ \mid E \mid}{\rho^{n}}\bigg)^{2-p}.$ Test \eqref{eq3.5} by $\eta=(v_{h}-k_{j})_{+} \xi^{q},$ integrating it over $(\bar{t}, t),~ t\in(\bar{t}, \bar{t}+8\tau_{1})$ and letting $h\rightarrow 0$, we arrive at
\begin{multline*}
\sup\limits_{\bar{t}<t<\bar{t}+8\tau_{1}}\int\limits_{B_{8\rho}(\bar{x})}(v-k_{j})^{2}_{+} \zeta^{q}dx +\iint\limits_{Q_{1}}
|\nabla (v-k_{j})_{+}|^{p}\zeta^{q} dxdt \leqslant\\
\leqslant \gamma\iint\limits_{Q_{1}\setminus \mathcal{D}_{j}} (v-k_{j})^{2}_{+}|\zeta_{t}|\zeta^{q-1}dxdt+
\gamma\iint\limits_{Q_{1}\setminus \mathcal{D}_{j}}\varPhi(x,t,(v-k_{j})_{+}|\nabla\zeta|)dxdt\leqslant \\
\leqslant \gamma \sigma^{-q}2^{\gamma j} \left(\tau_{1}^{-1}\iint\limits_{Q_{1}\setminus \mathcal{D}_{j}} (v-k_{j})^{2}_{+}dxdt+ \rho^{-p}\iint\limits_{Q_{1}\setminus \mathcal{D}_{j}} (v-k_{j})^{p}_{+}dxdt\right).
\end{multline*}
Above, we also used the following inequality, which is a consequence of our choices, condition ($ \varPhi_{\lambda}$), the fact that $v(x,t)\leqslant M \lambda(r)$ and $Q_{1}\subset Q_{\rho,\rho}(\bar{x},\bar{t})\subset Q_{2\rho, 2\rho}(x_{0},t_{0})$ :
$$
\varPhi\big(x,t,\frac{(v-k_{j})_{+}}{\rho}\big) \leqslant \bigg(\frac{v-k_{j}}{\rho}\bigg)^{p}_{+}\bigg(1+\max\limits_{Q_{\rho,\rho}(\bar{x},\bar{t})}a(x,t)\rho^{p-q}(M\lambda(r))^{q-p} \bigg)\leqslant
$$
$$
 \leqslant \bigg(\frac{v-k_{j}}{\rho}\bigg)^{p}_{+}\bigg(1+\max\limits_{Q_{2\rho,2\rho}(x_{0},t_{0})}a(x,t)\rho^{p-q}(M\lambda(r))^{q-p} \bigg)\leqslant \bigg(\frac{v-k_{j}}{\rho}\bigg)^{p}_{+}\times
 $$
 $$
\times\bigg(1+\gamma \mu(2\rho)(M\lambda(r))^{q-p} \bigg)\negthickspace \leqslant \negthickspace \bigg(\frac{v-k_{j}}{\rho}\bigg)^{p}_{+}\bigg(1+\gamma M^{q-p} \bigg) \negthickspace\leqslant \negthickspace \gamma \bigg(\frac{v-k_{j}}{\rho}\bigg)^{p}_{+},
$$
where $\quad(x,t)\in Q_{1}\setminus \mathcal{D}_{j}.$

Set $A_{j,k_{j}}:=\mathcal{D}_{j}\cap\{v\geqslant k_{j} \}$, then by the Sobolev embedding theorem from the previous  we obtain
\begin{multline*}
y_{j+1}=\negthickspace\negthickspace\iint\limits_{A_{j+1,k_{j+1}}} \negthickspace\negthickspace(v-k_{j})^{p}_{+}dxdt \leqslant \left(\,\,\,\iint\limits_{A_{j+1,k_{j+1}}} (v-k_{j})^{p\frac{n+2}{n}}_{+}\zeta^{q\frac{n+2}{n}}dxdt\right)^{\frac{n}{n+2}}\times\\
\times|A_{j,k_{j}}|^{\frac{2}{n+2}}\leqslant
\left(\sup\limits_{\bar{t}<t<\bar{t}+8\tau_{1}}\int\limits_{B_{8\rho}(\bar{x})}(v-k_{j})^{2}_{+}\zeta^{q}dx\right)^{\frac{p}{n+2}}\times
\qquad\qquad\qquad\\
\times\left(\iint\limits_{Q_{1}}|\nabla((v-k_{j})_{+}\zeta^{\frac{q}{p}})|^{p}\negthickspace\right)^{\frac{n}{n+2}}|A_{j,k_{j}}|^{\frac{2}{n+2}}
\leqslant\qquad\qquad\qquad\qquad
\\ \leqslant \gamma \sigma^{-\gamma}2^{j\gamma}\left(\frac{\tau_{1}^{-1}}{k^{p-2}}+\rho^{-p}\right)^{\frac{n+p}{n+2}} k^{-p(1-\frac{n}{n+2})} y_{j}^{1+\frac{p}{n+2}}~,j=0,1,2,...
\end{multline*}
Iterating the last inequality, we get that $\lim\limits_{j\rightarrow +\infty} y_{j} =0$, provided $k$ is chosen to satisfy
\begin{equation}\label{eq3.12}
k^{2} = \gamma \sigma^{-\gamma}\bigg(\frac{\tau_{1}^{-1}}{k^{p-2}} + \rho^{-p}\bigg)^{\frac{n+p}{p}}
\iint\limits_{Q_{1}\setminus \mathcal{D}_{0}}v^{p}dxdt.
\end{equation}
To estimate the integral on the right-hand side of \eqref{eq3.12}, we  test identity \eqref{eq3.5} by $\eta=\min(v_{h},M_{0})$. Integrating it over
$(\bar{t},t),t\in(\bar{t},\bar{t}+8\tau_{1})$ and letting $h\rightarrow 0$, for $v_{M_{0}}=\min(v,M_{0})$, we obtain
\begin{equation}\label{eq3.13}
\sup\limits_{\bar{t}<t<\bar{t}+8\tau_{1}}\int\limits_{B_{8\rho}(\bar{x})} \negthickspace v_{M_{0}}^{2} dx +\negthickspace\iint\limits_{Q_{1}}\negthickspace
\varPhi(x,t,|\nabla v_{M_{0}}|) dxdt \leqslant \gamma M_{0} N \lambda(r) |E|.
\end{equation}
Assumming that $k\geqslant (\rho^{p} /\tau_{1})^{\frac{1}{p-2}}= N\lambda(r)\dfrac{|E|}{\rho^{n}},$ from \eqref{eq3.12} and \eqref{eq3.13} by the Poincare inequality, using the fact that $v=v_{M_{0}}$ on $Q_{1}\setminus \mathcal{D}_{0}$, we obtain that
$$
M_{\sigma}^{2} \leqslant \gamma \sigma^{-\gamma}\rho^{-n-p}\iint\limits_{Q_{1}\setminus \mathcal{D}_{0}}v^{p}dxdt=
\gamma \sigma^{-\gamma}\rho^{-n-p}\iint\limits_{Q_{1}\setminus \mathcal{D}_{0}}v_{M_{0}}^{p}dxdt \leqslant
 $$
 $$
 \leqslant\gamma \sigma^{-\gamma}\rho^{-n}\iint\limits_{Q_{1}\setminus \mathcal{D}_{0}}|\nabla v_{M_{0}}|^{p}dxdt \leqslant
$$
\begin{equation}\label{eq3.14}
\leqslant \gamma\sigma^{-\gamma}\rho^{-n}\iint\limits_{Q_{1}}\varPhi(x,t,|\nabla v_{M_{0}}|) dxdt \leqslant \gamma\sigma^{-\gamma} M_{0} N \lambda(r)\frac{|E|}{\rho^{n}}.
\end{equation}
Using the Young inequality, we obtain for every $\varepsilon\in(0,1)$
\begin{equation*}
M_{\sigma} \leqslant \varepsilon M_{0} +\gamma \sigma^{-\gamma} \varepsilon^{-\gamma} N \lambda(r)\frac{|E|}{\rho^{n}},
\end{equation*}
from which, by iteration, the required inequality \eqref{eq3.10} follows.

The proof of \eqref{eq3.11} is completely similar, we also use the inequality, which is a consequence of our choices
\begin{multline*}
\varPhi\left(x,t,\frac{(w-k_{j})_{+}}{\rho_{0}}\right)\negthickspace \leqslant \negthickspace \bigg(\frac{w-k_{j}}{\rho_{0}}\bigg)^{p}_{+}\negthickspace\bigg(\negthickspace1+
\negthickspace\max\limits_{Q_{\rho_{0},\rho_{0}}(\bar{x},\bar{t})}\negthickspace a(x,t)\rho_{0}^{p-q}(M\lambda(r))^{q-p} \bigg)\\
\leqslant \bigg(\frac{w-k_{j}}{\rho_{0}}\bigg)^{p}_{+}\bigg(1+\max\limits_{Q_{8\rho_{0}, 8\rho_{0}}(\bar{x},\bar{t})}a(x,t)\rho_{0}^{p-q}(M\lambda(r))^{q-p} \bigg)\leqslant \\
\leqslant\bigg(\frac{w-k_{j}}{\rho_{0}}\bigg)^{p}_{+}\bigg(1+\gamma \mu(8\rho_{0})(M\lambda(r))^{q-p} \bigg)\leqslant\\
\leqslant \bigg(\frac{w-k_{j}}{\rho_{0}}\bigg)^{p}_{+}\bigg(1+\gamma M^{q-p} \bigg) \leqslant \gamma \bigg(\frac{w-k_{j}}{\rho_{0}}\bigg)^{p}_{+}, \quad(x,t)\in Q_{2}\setminus \mathcal{D}_{j}.
\end{multline*}
This completes the proof of the lemma.
\end{proof}

\begin{lemma}\label{lem3.2}
{\it  There exist numbers $\varepsilon_{1}, \alpha_{1}, \delta_{1}\in(0, 1)$ depending only on the data  such that
\begin{equation}\label{eq3.15}
\left| \left\{B_{4\rho}(\bar{x}): v (\cdot, t_{1})\leqslant \varepsilon_{1}N\lambda(r)\frac{\mid E \mid}{\rho^{n}} \right\} \right| \leqslant(1-\alpha_{1})\mid B_{4\rho}(\bar{x}) \mid
\end{equation}
for some time level $t_{1}\in (\bar{t}+\delta_{1}\tau_{1},~ \bar{t}+\tau_{1}),$
\begin{equation}\label{eq3.16}
\left| \left\{B_{4\rho_{0}}(\bar{x}): w (\cdot, t_{2})\leqslant \varepsilon_{1}N\lambda(r)\frac{\mid E \mid}{\rho_{0}^{n}} \right\} \right| \leqslant(1-\alpha_{1})\mid B_{4\rho_{0}}(\bar{x}) \mid
\end{equation}
for some time level $t_{2}\in (\bar{t}+\delta_{1}\tau_{2},~ \bar{t}+\tau_{2}),$
}
\end{lemma}

\begin{proof} Let $\zeta_{1}(x)\in C^{\infty}_{0}(B_{3\rho}(\bar{x}))$, $0\leqslant\zeta_{1}(x)\leqslant 1$, $\zeta_{1}(x)=1$ in $B_{2\rho}(\bar{x})$, $\mid \nabla\zeta_{1}(x)\mid\leqslant\dfrac{1}{\rho}.$ Testing \eqref{eq3.5} by $\eta=v_{h}-N\lambda(r)\zeta^{q}_{1}(x),$ integrating it over $(\bar{t}, \bar{t}+\tau_{1})$ and letting $h\rightarrow 0$, we obtain
$$
\frac{N^{2}}{2}[\lambda(r)]^{2}\mid E\mid + \frac{1}{2}\int\limits_{B_{8\rho}(\bar{x})} v^{2}(x, \bar{t}+\tau_{1})dx+$$
$$
+\gamma^{-1}\int\limits^{\bar{t}+\tau_{1}}\limits_{\bar{t}}\int\limits_{B_{8\rho}(\bar{x})} \varPhi(x, t, \mid \nabla v\mid) dxdt\leqslant
 N\lambda(r) \int\limits_{B_{8\rho}(\bar{x})} v(x, \bar{t}+\tau_{1})\zeta^{q}_{1}(x)dx+
$$
\begin{equation}\label{eq3.17}
+\gamma N \frac{\lambda(r)}{\rho} \int\limits_{\bar{t}}\limits^{\bar{t}+\tau_{1}}\int\limits_{B_{3\rho}(\bar{x})\setminus B_{2\rho}(\bar{x})} \varphi(x, t, \mid \nabla v \mid)\zeta^{q-1}_{1}(x)dxdt=I_{1}+I_{2}.
\end{equation}

Let us estimate the terms on the right-hand side of \eqref{eq3.17}. By Lemma \ref{lem3.1} we obtain
$$
I_{1}\leqslant \varepsilon_{1} N^{2}[\lambda(r)]^{2} \mid E\mid+\qquad\qquad\qquad\qquad\qquad\qquad\qquad\qquad\qquad\qquad
$$
\begin{equation}\label{eq3.18}
+\gamma N^{2}[\lambda(r)]^{2}\frac{\mid E \mid}{\rho^{n}} \left|\left\{ B_{4\rho}(\bar{x}):v(\cdot, \bar{t}+\tau_{1}) \geqslant \varepsilon_{1}N\lambda(r)\frac{\mid E \mid}{\rho^{n}}\right\} \right|.
\end{equation}

Let $\zeta_{2}(x)\in C^{\infty}(\mathbb{R}^{n})$, $0\leqslant\zeta_{2}(x)\leqslant 1$, $\zeta_{2}(x)=1$ in $B_{3\rho}(\bar{x})\setminus B_{2\rho}(\bar{x}),$ $\zeta_{2}(x)=0$ for $x\in B_{\frac{3}{2}\rho}(\bar{x})$ and for $x\in \mathbb{R}^{n}\setminus B_{4\rho}(\bar{x}),$  $\mid \nabla \zeta_{2}(x)\mid~ \leqslant \gamma\rho^{-1}.$ Using the Young inequality with $\varepsilon=\varepsilon_{0}N\lambda(r)\dfrac{\mid E\mid}{\rho^{n+1}},$ where $\varepsilon_{0}\in (0, 1)$ to be determined later, we obtain
\begin{multline}\label{eq3.19}
I_{2}\leqslant\gamma N\varepsilon^{-1}\frac{\lambda(r)}{\rho} \int\limits_{\bar{t}}\limits^{\bar{t}+\tau_{1}} \int\limits_{B_{4\rho}(\bar{x})\setminus B_{\frac{3}{2}}(\bar{x})} \varPhi(x, t, \mid \nabla v \mid)\mid \nabla v \mid \zeta_{2}^{q}(x)dxdt+\\
+\gamma N\frac{\lambda(r)}{\rho}\int\limits_{\bar{t}}\limits^{\bar{t}+\tau_{1}}\int\limits_{B_{4\rho}(\bar{x})} \varphi(x, t, \varepsilon)dxdt=I_{3}+I_{4}.
\end{multline}
By condition ($\varPhi_{\lambda}$) we have
\begin{equation*}
\max\limits_{Q_{2\rho, 2\rho}(x_{0},t_{0})} \varphi\left(x,t, N\lambda(r)\frac{ \mid E \mid}{\rho^{n+1}}\right)\leqslant \gamma \rho^{1-p}\bigg(N\lambda(r)\frac{ \mid E \mid}{\rho^{n}}\bigg)^{p-1},
\end{equation*}
so
\begin{multline}\label{eq3.20}
I_{4}\leqslant \gamma N\frac{\lambda(r)}{\rho} \varepsilon_{0}^{p-1}\int\limits_{\bar{t}}\limits^{\bar{t}+\tau_{1}}\int\limits_{B_{4\rho}(\bar{x})} \varphi\left(x, t,N\lambda(r) \frac{\mid E \mid}{\rho^{n+1}}\right)dx\,dt\leqslant\\
\leqslant\gamma N\frac{\lambda(r)}{\rho}\varepsilon_{0}^{p-1}\max\limits_{Q_{2\rho, 2\rho}(x_{0},t_{0})} \varphi\left(x,t, N\lambda(r)\frac{ \mid E \mid}{\rho^{n+1}}\right)\mid B_{\rho}(\bar{x})\mid \tau_{1} \leqslant\\
\leqslant \gamma \varepsilon_{0}^{p-1} N^{2}[\lambda(r)]^{2} \mid E \mid.
\end{multline}

To estimate $I_{3}$ we test \eqref{eq3.5} by $\eta=v_{h}\zeta^{q}_{2}(x),$ integrating it over $(\bar{t}, \bar{t}+\tau_{1})$ and letting $h\rightarrow 0$, we arrive at
$$I_{3}\leqslant \gamma \varepsilon^{-1}N\frac{\lambda(r)}{\rho} \int\limits_{\bar{t}}\limits^{\bar{t}+\tau_{1}}\int\limits_{B_{4\rho}(\bar{x})\setminus B_{\frac{3}{2}\rho}(\bar{x})} \varPhi(x, t, \frac{v}{\rho}) dxdt.$$

From this, by condition $(g_{\lambda})$ and Lemma \ref{lem3.1} we obtain
$$
I_{3}\negthickspace\leqslant\negthickspace\gamma\varepsilon^{-1}N\frac{\lambda(r)}{\rho} \negthickspace\negthickspace \int\limits^{\bar{t}+\tau_{1}}\limits_{\bar{t}} \negthickspace \negthickspace\int\limits_{B_{4\rho}(\bar{x})\setminus B_{\frac{3}{2}\rho}(\bar{x})}\negthickspace\negthickspace \negthickspace \negthickspace\varPhi^{+}_{Q_{2\rho, 2\rho}(x_{0},t_{0})}\bigg( \frac{v}{\rho}\bigg)dxdt
\negthickspace\leqslant\negthickspace \gamma \frac{\varepsilon_{1}}{\varepsilon_{0}}  N^{2}[\lambda(r)]^{2}  \mid E\mid\negthickspace+
$$
\begin{equation}\label{eq3.21}
 + \gamma  \frac{N^{2}[\lambda(r)]^{2} |E|}{\varepsilon_{0}\tau_{1}\rho^{n}} \negthickspace \int\limits^{\bar{t}+\tau_{1}}\limits_{\bar{t}} \left| \left\{ B_{4\rho}(\bar{x})\setminus  B_{\frac{3}{2}}(\bar{x})\negthickspace: v(\cdot, t)\negthickspace \geqslant\negthickspace \varepsilon_{1} N\lambda(r) \frac{ E}{\rho^{n}} \right\}\right| dt.
\end{equation}

Collecting estimates \eqref{eq3.17}-\eqref{eq3.21}, we arrive at
\begin{multline*}
\frac{1}{2} N^{2}[\lambda(r)]^{2} \mid E \mid \leqslant \gamma \left(\varepsilon_{1}+\varepsilon_{0}^{p-1}+\frac{\varepsilon_{1}}{\varepsilon_{0}}\right) N^{2}[\lambda(r)]^{2}\mid E\mid +\\
+\gamma  N^{2}[\lambda(r)]^{2} \frac{\mid E\mid}{\rho^{n}} \left|\left\{B_{4\rho}(\bar{x}):v(\cdot, \bar{t}+\tau_{1})\geqslant\varepsilon_{1}N\lambda(r)\frac{\mid E\mid}{\rho^{n}}\right\}\right|+\\
\gamma N^{2}[\lambda(r)]^{2} \frac{\mid E\mid}{\varepsilon_{0}\tau_{1}\rho^{n}} \int\limits_{\bar{t}}\limits^{\bar{t}+\tau_{1}} \left|\left\{ B_{4\rho}(\bar{x}):v(\cdot, t)\geqslant\varepsilon_{1} N \frac{\mid E\mid}{\rho^{n}} \right\}\right| dt
\end{multline*}

Choose $\varepsilon_{0}$ such that $\gamma\varepsilon_{0}^{p-1}=\dfrac{1}{8},$ and $\varepsilon_{1}$ such that $\gamma \varepsilon_{1}(1+\dfrac{1}{\varepsilon_{0}})=\dfrac{1}{8},$ from the previous  we obtain
$$\gamma^{-1}\rho^{n}\leqslant \left|\left\{B_{4\rho}(\bar{x}):v(\cdot, \bar{t}+\tau_{1})\geqslant\varepsilon_{1}N\lambda(r)\frac{\mid E\mid}{\rho^{n}}\right\}\right|+$$
$$+\frac{1}{\tau_{1}}\int\limits^{\bar{t}+\tau_{1}}\limits_{\bar{t}}\left|\left\{B_{4\rho}(\bar{x}):v(\cdot, t)\geqslant\varepsilon_{1}N\lambda(r)\frac{\mid E\mid}{\rho^{n}}\right\}\right| dt.$$
From this, we conclude that at least one of the following two inequalities holds
$$\left|\left\{B_{4\rho}(\bar{x}):v(\cdot, \bar{t}+\tau_{1})\geqslant\varepsilon_{1}N\lambda(r)\frac{\mid E\mid}{\rho^{n}}\right\}\right|\geqslant\frac{1}{2\gamma}\mid B_{4\rho}(\bar{x})\mid,$$
$$\int\limits_{\bar{t}}\limits^{\bar{t}+\tau_{1}}\left|\left\{B_{4\rho}(\bar{x}):v(\cdot, t)\geqslant\varepsilon_{1}N\lambda(r)\frac{\mid E\mid}{\rho^{n}}\right\}\right|\geqslant\frac{1}{2\gamma} \tau_{1}\mid B_{4\rho}(\bar{x})\mid.$$
From the second one it follows that there exists $t_{1}\in (\bar{t}+\frac{1}{4\gamma}\tau_{1}, \bar{t}+\tau_{1})$ such that
$$\big|\big\{B_{4\rho}(\bar{x}):v(\cdot, t_{1})\geqslant\varepsilon_{1}N\lambda(r)\frac{\mid E\mid}{\rho^{n}}\big\}\big|\geqslant\frac{1}{4\gamma-1}\mid B_{4\rho}(\bar{x})\mid,$$
indeed, if not, then
$$(1-\frac{1}{2\gamma})\tau_{1}\mid B_{4\rho}(\bar{x})\mid<\int\limits^{\bar{t}+\tau_{1}}\limits_{\bar{t}+\frac{1}{4\gamma}\tau_{1}}
\left|\left\{B_{4\rho}(\bar{x}):v(\cdot,t)\leqslant\varepsilon_{1}N\lambda(r)\frac{\mid E\mid}{\rho^{n}}\right\}\right| dt \leqslant$$
$$\leqslant \int\limits^{\bar{t}+\tau_{1}}\limits_{\bar{t}}\left|\left\{B_{4\rho}(\bar{x}):v(\cdot,t)\leqslant\varepsilon_{1}N\lambda(r)\frac{\mid E\mid}{\rho^{n}}\right\}\right|dt\leqslant(1-\frac{1}{2\gamma})\tau_{1}\mid B_{4\rho}(\bar{x})\mid,$$
reaching a contradiction. This proves inequality \eqref{eq3.15}.

The proof of \eqref{eq3.16} is completely similar, we also use the inequality,
which is a consequence of our choices
\begin{equation*}
 \varPhi^{+}_{Q_{8\rho_{0}, 8\rho_{0}}(\bar{x},\bar{t})}\bigg( N\lambda(r)\frac{ \mid E \mid}{\rho_{0}^{n+1}}\bigg)\leqslant \frac{\gamma}{ \rho_{0}^{p}}\bigg(N\lambda(r)\frac{ \mid E \mid}{\rho_{0}^{n}}\bigg)^{p},
\end{equation*}
this completes the proof of the lemma.
\end{proof}

\section{Expansion of positivity}\label{Sect4}
The following theorem will be used in the sequel which is an expansion of positivity result. In the case of the p-Laplace evolution equation this result was proved by DiBenedetto, Gianazza and Vespri \cite{DiBGiVes3}, in the logarithmic case  this theorem was proved in \cite{BurchSkrPotAn}.

\begin{theorem}\label{th4.1}
 Let \,\,$u$\,\, be a \,\,non-negative \,\,bounded\,\, weak\,\, solution\,\, to \\ Eq. \eqref{eq1.1} and let conditions \eqref{eq1.2}-- \eqref{eq1.4} be fulfilled. Fix point $(x_{0},t_{0})\in \Omega_{T}$ such that $a(x_{0},t_{0})=0$, and let  for some $\rho>0$, for some $0<N\leqslant M$ and some $\delta\in (0, 1),$
$$ Q_{\rho,\theta}(y,s)\subset Q_{ 2\rho,2\rho}(x_{0},t_{0})\subset Q_{ 8\rho, 8\rho}(x_{0},t_{0}) \subset \Omega_{T},\quad \theta=\delta \rho^{p}  \big(N\lambda(\rho)\big)^{2-p}.$$
Assume also that
\begin{equation}\label{eq4.1}
\mid\{B_{\rho}(y):u(\cdot,s)\leqslant N \lambda(\rho) \}\mid\leqslant(1-\alpha)\mid B_{\rho}(y)\mid,
\end{equation}
for some $\alpha \in (0,1)$.
Then there exist $\sigma_{0}\in(0, 1)$ and $1 <\bar{C}_{1}<\bar{C}_{2}$ depending only upon the data and $\alpha, \delta$ such that either

\begin{equation}\label{eq4.2}
\sigma_{0}N \lambda(\rho) \leqslant \rho,
\end{equation}
or
\begin{equation}\label{eq4.3}
u(x, t)\geqslant\sigma_{0}N \lambda(\rho) ,\quad \text{for all}\quad (x,t)\in B_{2\rho}(y)\times (s+ \bar{C}_{1}\theta, s+ \bar{C}_{2}\theta).
\end{equation}

\end{theorem}

\textbf{Proof of Theorem \ref{th4.1}}

We will suppose that inequality \eqref{eq4.2} is violated, i.e.
\begin{equation}\label{eq4.4}
C_{*}N \lambda(\rho) \geqslant \rho,
\end{equation}
where $C_{*}$ is a  positive number to be chosen later depending on the known data only. By our assumptions and by ($\varPhi_{\lambda}$)
condition
\begin{multline*}
\bigg(\frac{N\lambda(\rho)}{\rho}\bigg)^{p-2}\leqslant \psi^{+}_{Q_{\rho,\rho}(y, s)}\bigg(\frac{N \lambda(\rho)}{\rho}\bigg)\leqslant \\
\leqslant\psi^{+}_{Q_{2\rho,2\rho}(x_{0}, t_{0})}\bigg(\frac{N \lambda(\rho)}{\rho}\bigg)\leqslant  \gamma\bigg(\frac{N\lambda(\rho)}{\rho}\bigg)^{p-2},
\end{multline*}
therefore inequality \eqref{eq4.1} and Lemma \ref{lem2.3} with $r$ replaced by $\rho$, $N$ replaced by $ N \lambda(\rho) e^{-\tau}, \tau>0$ implies that
\begin{equation}\label{eq4.5}
\{B_{\rho}(y)\negthickspace:u(\cdot,s+\bar{\delta}_{0}\rho^{p}(N \lambda(\rho)e^{-\tau})^{2-p}\negthickspace\leqslant \negthickspace \varepsilon_{0}Ne^{-\tau}\}\leqslant\left(1-\frac{\alpha^{2}}{2}\right) B_{\rho}(y),
\end{equation}
for all $\tau>0$ and $\bar{\delta}_{0}=\gamma^{-1}\delta_{0}$, $\delta_{0}$ is the number defined in Lemma \ref{lem2.3}.

Following \cite{DiBGiVe1}, we introduce the change of variables and the new unknown function: \\$x=y+z\rho,t=s+\bar{\delta}_{0}\rho^{p}(N \lambda(\rho) e^{-\tau})^{2-p},\qquad h(z,\tau)=\dfrac{e^{\tau}}{N \lambda(\rho)}u(x,t)$.

Inequality \eqref{eq4.5} transforms into $h$ as
\begin{equation}\label{eq4.6}
\mid\{ B_{1} : h\leqslant \varepsilon_{0}\}\mid \leqslant \left(1-\frac{\alpha^{2}}{2}\right) \mid B_{1} \mid, \quad B_{1}:=B_{1}(0),
\end{equation}
for all $\tau >0$. Since $h>0$, the formal differentiation gives
\begin{equation}\label{eq4.7}
h_{\tau}= h +(p-2)\bar{\delta}_{0}\rho^{p}\bigg(\frac{e^{\tau}}{ N \lambda(\rho)}\bigg)^{p-1}u_{t} =\textrm{div}\mathbb{\bar{A}}(x, t, \nabla h) + h,
\end{equation}
where $\mathbb{\bar{A}}$ satisfies the inequalities
\begin{equation}\label{eq4.8}
\begin{aligned}
&\mathbb{\bar{A}}(x, t, \nabla h) \nabla h \geqslant (p-2) \bar{\delta}_{0}K_{1} \bigg(\mid \nabla h \mid^{p}+ \bar{a}(z,\tau) \mid \nabla h \mid^{q}\bigg),\\
&\mid \mathbb{\bar{A}}(x, t, \nabla h) \mid \leqslant (p-2) \bar{\delta}_{0}K_{2} \bigg(\mid \nabla h \mid^{p-1}+ \bar{a}(z,\tau) \mid \nabla h \mid^{q-1}\bigg),
\end{aligned}
\end{equation}
where~$\bar{a}(z,\tau)=\bigg(\dfrac{N \lambda(\rho)}{e^{\tau}\rho}\bigg)^{q-p} a(y+z\rho,s+\bar{\delta}_{0}\rho^{p}(N \lambda(\rho)e^{-\tau})^{2-p})~.$

\begin{lemma}\label{lem4.1}
{\it For every $\nu$ there exists $s_{*}>1$ depending only on the data, $\alpha$, $\bar{\delta}_{0}$ and $\nu$ such that
\begin{equation}\label{eq4.9}
\left| \left\{Q_{*} : h \leqslant \frac{\varepsilon_{0}}{2^{s_{*}}} \right\} \right| \leqslant \nu \mid Q_{*} \mid,
\end{equation}
where $Q_{*}:=B_{1}\times \bigg(\bigg(\dfrac{2^{s_{*}}}{\varepsilon_{0}}\bigg)^{p-2},2\bigg(\dfrac{2^{s_{*}}}{\varepsilon_{0}}\bigg)^{p-2}\bigg).$
}
\end{lemma}
\begin{proof}Using Lemma \ref{lem2.1} with $k=k_{s+1}$, $l=k_{s}$, $k_{s}= \dfrac{\varepsilon_{0}}{2^{s}}$, due to \eqref{eq4.6} we
obtain for every $1\leqslant s \leqslant s_{*}-1$
\begin{equation*}
(k_{s}-k_{s+1})\mid A_{s+1}(\tau) \mid \leqslant \gamma \alpha^{2} \negthickspace\negthickspace \int\limits_{A_{s}(\tau)\setminus A_{s+1}(\tau)}\negthickspace \negthickspace\negthickspace\mid\nabla  h\mid dz,\,\,  A_{s}(\tau):=\{B_{1}: h\leqslant k_{s}\},
\end{equation*}
for all $\tau >0$. Integrating this inequality with respect to $\tau \negthickspace\in\negthickspace (  k_{s_{*}}^{2-p},2 k_{s_{*}}^{2-p})$ and using the H\"{o}lder inequality, we have
\begin{equation}\label{eq4.10}
(k_{s}-k_{s+1})^{\frac{p}{p-1}}\mid A_{s+1} \mid^{\frac{p}{p-1}} \leqslant \gamma(\alpha) \left(\iint\limits_{A_{s}}\mid\nabla h \mid^{p}\right)^{\frac{1}{p-1}} \negthickspace\mid A_{s}\setminus A_{s+1} \mid,
\end{equation}
where $A_{s}:=\int\limits_{k_{s_{*}}^{2-p}}^{2 k_{s_{*}}^{2-p}}A_{s}(\tau)d\tau.$
To estimate the first term on the right-hand side of \eqref{eq4.10} we use Lemma \ref{lem2.2} with $k=k_{s}, \zeta \in C_{0}^{\infty}(\bar{Q}_{*})$, $\bar{Q}_{*}=B_{2}\times(\frac{1}{2}k_{s_{*}}^{2-p},4k_{s_{*}}^{2-p})$, $0\leqslant\zeta\leqslant 1$, $\zeta=1$ in $Q_{*}$, $\mid \nabla \zeta \mid \leqslant 2$, $ \mid \zeta_{\tau} \mid \leqslant 2 k_{s_{*}}^{p-2}$. Due to \eqref{eq4.8} we have
\begin{multline*}
\iint\limits_{A_{s}}\mid \nabla h \mid^{p} dx d\tau \leqslant \gamma \iint\limits_{\bar{Q}_{*}} (h-k_{s})^{2}_{-} \mid \zeta_{\tau} \mid dx d\tau + \\
+\gamma \iint\limits_{\bar{Q}_{*}} (h-k_{s})^{p}_{-} \mid \nabla \zeta \mid^{p} dx d\tau + \gamma \iint\limits_{\bar{Q}_{*}}\bar{a}(z,\tau) (h-k_{s})^{q}_{-} \mid \nabla \zeta \mid^{q} dx d\tau \leqslant\\
\leqslant \gamma k_{s}^{p} \bigg(1+ k_{s}^{q-p}\max\limits_{\bar{Q}_{*}}\bar{a}(z,\tau) \bigg) \mid Q_{*} \mid.
\end{multline*}
If $ C_{*} \geqslant e^{4(\frac{2^{s_{*}}}{\varepsilon_{0}})^{p-2}}$, then by \eqref{eq4.1} $\bar{\delta}_{0}\rho^{p}(N \lambda(\rho) e^{-\tau})^{2-p}\leqslant \rho $, and therefore by condition ($\varPhi_{\lambda}$)
\begin{multline*}
k_{s}^{q-p}\max\limits_{\bar{Q}_{*}}\bar{a}(z,\tau) \leqslant \bigg(\frac{M\varepsilon_{0}\lambda(\rho)}{\rho}\bigg)^{q-p} \max\limits_{Q_{2\rho, 2\rho}(x_{0},t_{0})} a(x,t) \leqslant \\
\leqslant A\bigg(\frac{M\varepsilon_{0}\lambda(\rho)}{\rho}\bigg)^{q-p}\mu(2\rho)(2\rho)^{q-p} \leqslant \gamma \lambda(\rho)^{q-p} \mu(\rho) = \gamma,
\end{multline*}
so
\begin{equation}\label{eq4.11}
\iint\limits_{A_{s}}\mid \nabla h \mid^{p} dx d\tau \leqslant \gamma k_{s}^{p} \mid Q_{*} \mid.
\end{equation}
Combining estimates \eqref{eq4.10} and \eqref{eq4.11}, we obtain
\begin{equation*}
\mid A_{s+1} \mid^{\frac{p}{p-1}} \leqslant \mid Q_{*} \mid ^{\frac{1}{p-1}} \gamma \mid A_{s}\setminus A_{s+1} \mid.
\end{equation*}
Summing up this inequality for $1\leqslant s\leqslant s_{*}-1 $, we conclude that
\begin{equation*}
\mid A_{s_{*}} \mid \leqslant \gamma (s_{*}-1)^{-\frac{p-1}{p}} \mid Q_{*} \mid,
\end{equation*}
choosing $s_{*}$ from the condition $\gamma (s_{*}-1)^{-\frac{p-1}{p}} \leqslant \nu$ we arrive at the required \eqref{eq4.9}, which completes the proof of the lemma.
\end{proof}
Using the fact that $k^{q-p}_{s_{*}} \max\limits_{\bar{Q}_{*}}\bar{a}(z,\tau) \leqslant \gamma$, by Lemma \ref{lem2.4} from \eqref{eq4.9} we obtain that
$$ h(x,\tau) \geqslant \frac{\varepsilon_{0}}{2^{s_{*}+1}},\quad (x,\tau) \in B_{\frac{1}{2}}\times \bigg(\frac{5}{4}\bigg(\frac{2^{s_{*}}}{\varepsilon_{0}}\bigg)^{p-2},\frac{7}{4}\bigg(\frac{2^{s_{*}}}{\varepsilon_{0}}\bigg)^{p-2}\bigg).$$
This inequality can  be rewritten in terms of function $u$ as follows
$$u(x,t)\geqslant \varepsilon_{0}e^{-2(\frac{2^{s_{*}}}{\varepsilon_{0}})^{p-2}}2^{-s_{*}-1} N \lambda(\rho) ,$$
for all $(x,t) \in B_{\frac{\rho}{2}}(y)\times\bigg( s+\bar{\delta}_{0}e^{\frac{5}{4}(\frac{2^{s_{*}}}{\varepsilon_{0}})^{p-2}}\rho^{p} \big(N\lambda(\rho)\big)^{2-p} ,s+\bar{\delta}_{0}e^{\frac{7}{4}(\frac{2^{s_{*}}}{\varepsilon_{0}})^{p-2}}\rho^{p} \\ \big(N\lambda(\rho)\big)^{2-p}\bigg).$\\
This proves Theorem \ref{th4.1} with $\bar{C}_{1}=\bar{\delta}_{0}e^{\frac{5}{4}(\frac{2^{s_{*}}}{\varepsilon_{0}})^{p-2}}$ and
$\bar{C}_{2}=\bar{\delta}_{0}e^{\frac{7}{4}(\frac{2^{s_{*}}}{\varepsilon_{0}})^{p-2}}$.

\section{Harnack's inequality, proof of Theorem \ref{th1.2}}\label{Sect5}
Fix $(x_{0},t_{0})\in \Omega_{T} $ such that $a(x_{0},t_{0})=0$ and for $\tau \in(0,1)$ construct the cylinder $Q_{\tau}:=B_{\rho}(x_{0})\times (t_{0}- (\tau\rho)^{p}\big(u_{0} \lambda_{1}(\rho)\big)^{2-p}, t_{0})$, $u_{0}:=u(x_{0}, t_{0})$. Following Krylov and Safonov, we consider the equation
$$ M_{\tau}= N_{\tau},\quad M_{\tau}:= \sup\limits_{Q_{\tau}}u, \quad N_{\tau}:= \frac{1}{2} u_{0} (1-\tau)^{-n} \frac{\lambda_{1}(\rho)}{\lambda_{1}\big((1-\tau)\rho\big)},$$
$$\lambda_{1}(\rho)=\lambda(\rho)[\mu(\rho)]^{-n}.$$
Let $\tau_{0}$ be the maximal root of the above equation and $u(y,s)= N_{\tau_{0}}$. Let $r=\dfrac{1-\tau_{0}}{2}\rho$ and set $\theta=\dfrac{r^{2}}{\psi^{+}_{Q_{4r, 4r}(y,s)}\bigg(\dfrac{N_{\tau_{0}}}{r}\bigg)}$, since
$$\psi^{+}_{Q_{4r, 4r}(y,s)}\bigg(\dfrac{N_{\tau_{0}}}{r}\bigg) \geqslant \bigg(\frac{N_{\tau_{0}}}{r}\bigg)^{p-2}\geqslant \big(u_{0}\lambda_{1}(\rho)\big)^{p-2},$$
we have an inclusion $Q_{r,\theta}(y,s)\subset Q_{\frac{1+\tau_{0}}{2}}$, so by \eqref{eq1.13} there holds
$$\sup\limits_{Q_{r,\theta}(y,s)} u \leqslant 2^{n} u_{0} (1-\tau_{0})^{-n} \frac{\lambda_{1}(\rho)}{\lambda_{1}\big(\frac{1-\tau_{0}}{2}\rho)}=2^{n} N_{\tau_{0}} \frac{\lambda_{1}(2r)}{\lambda_{1}(r)}\leqslant 2^{n+b_{1}} N_{\tau_{0}}.$$

Further we will assume that inequality \eqref{eq1.15} is violated, i.e.
\begin{equation}\label{eq5.1}
u_{0}\geqslant C \frac{\rho}{\lambda_{1}(\rho)},
\end{equation}
with some $C >0$ to be determined later depending only on the data.

Claim $1$. There exists number $\nu >0$ depending only on the data such that
$$\left| \left\{Q^{-}_{r,\theta}(y,s): u \geqslant \frac{N_{\tau_{0}}}{2} \right\}\right| \geqslant \nu [\mu(r)]^{-n} \mid Q^{-}_{r,\theta}(y,s)\mid.$$
Indeed, if not, we apply Lemma \ref{lem2.4} for the function $2^{n+b_{1}}N_{\tau_{0}} -u$ with the choices
$$N= (2^{n+b_{1}}-\frac{1}{2})N_{\tau_{0}},\quad \xi_{0}=\dfrac{2^{n+b_{1}}-\frac{3}{4}}{2^{n+b_{1}}-\frac{1}{2}},\quad \nu=\gamma^{-1} (1-\xi_{0})^{2+nq},$$

\noindent condition \eqref{eq5.1} implies that $\theta \leqslant r$, therefore we conclude that $u(y,s)\leqslant \frac{3}{4}N_{\tau_{0}}$, reaching a contradiction, which proves the claim.

Claim $2$. There exists time level $\bar{s}\in (s-(1-\dfrac{\nu}{2}[\mu(r)]^{-n})\theta,s)$ such that
\begin{equation}\label{eq5.2}
\left| \left\{ B_{r}(y) : u(\cdot,\bar{s}) \geqslant \frac{N_{\tau_{0}}}{2}\lambda(r) \right\} \right| \geqslant \frac{\nu[\mu(r)]^{-n}}{2-\nu[\mu(r)]^{-n}}  \mid B_{r}(y) \mid.
\end{equation}
If not and if inequality \eqref{eq5.2} is violated for all $t\in (s-(1-\dfrac{\nu}{2}[\mu(r)]^{-n})\theta,s)$, then by Claim $1$ we have
\begin{multline*}
(1-\nu [\mu(r)]^{-n}) \mid Q^{-}_{r,\theta}(y,s) \mid <\negthickspace\left| \left\{ Q^{-}_{r,(1-\frac{\nu}{2}[\mu(r)]^{-n})\theta}(y,s): u\negthickspace\leqslant\negthickspace\frac{N_{\tau_{0}}}{2}\lambda(r) \right\}  \right| \negthickspace \leqslant \\ \leqslant \left|\left\{Q^{-}_{r,\theta}(y,s)\negthickspace:u \leqslant \frac{N_{\tau_{0}}}{2} \right\}\right| \leqslant (1-\nu [\mu(r)]^{-n}) \mid Q^{-}_{r,\theta}(y,s)\mid,
\end{multline*}
and reach a contradiction. This proves inequality \eqref{eq5.2}.

First we assume  that
$$\max\limits_{Q_{4r,4r}(y,\bar{s})} a(x,t)\leqslant 4A \mu(4r) (4r)^{q-p}$$
and construct the solution $v(x,t)=v_{r,\frac{1}{2}N_{\tau_{0}}}(x,t,y,\bar{s})$ with
$$N=\frac{N_{\tau_{0}}}{2}\quad \text{and} \quad E=E(\bar{s}):=\left\{ B_{r}(y) : u(\cdot,\bar{s}) \geqslant \frac{N_{\tau_{0}}}{2}\lambda(r)\right\}$$
 of the problem \eqref{eq3.2}-\eqref{eq3.4} in $Q_{1}=Q^{+}_{8\rho,8\tau_{1}}(y,\bar{s})$ where
$\tau_{1}=\rho^{p}\bigg(\dfrac{N_{\tau_{0}}}{2}\lambda(r)\\ \dfrac{\mid E(\bar{s}) \mid}{\rho^{n}}\bigg)^{2-p}$.
Inequality \eqref{eq3.15} of Lemma \ref{lem3.2} yields
\begin{equation*}
\left|\left\{B_{4\rho}(y): v(\cdot,t_{1}) \leqslant \varepsilon_{1}N_{\tau_{0}}\lambda(r) \frac{\mid E(\bar{s})\mid}{\rho^{n}}\right\}\right| \leqslant (1-\alpha_{1}) \big| B_{4\rho}(y) \big|,
\end{equation*}
for some time level $t_{1}\in (\bar{s} +\delta_{1}\tau_{1},\bar{s}+ \tau_{1})$ and the numbers $\varepsilon_{1},\alpha_{1},\delta_{1} \in (0,1)$ depended only on the data.

By \eqref{eq5.2} $|E(\bar{s})| \geqslant \dfrac{\nu}{2}[\mu(r)]^{-n} |B_{r}(y)|$, therefore, since $u\geqslant v$ on the parabolic boundary of $Q_{1}$, by the monotonicity condition \eqref{eq1.14}  we obtain
\begin{multline}\label{eq5.3}
\left|\left\{B_{4\rho}(y) :u(\cdot,t_{1})\leqslant  \varepsilon_{1}\frac{\nu}{2}N_{\tau_{0}}\lambda(r)[\mu(r)]^{-n}\bigg(\frac{r}{\rho}\bigg)^{n} \right\}\right| \leqslant\\
 \leqslant\left|\left\{B_{4\rho}(y) :u(\cdot,t_{1})\leqslant  \varepsilon_{1}N_{\tau_{0}}\lambda(r)\frac{\mid E(\bar{s})\mid}{\rho^{n}} \right\}\right|\leqslant \\ \leqslant \left|\left\{B_{4\rho}(y) :v(\cdot,t_{1})\leqslant  \varepsilon_{1}N_{\tau_{0}}\lambda(r)\frac{\mid E(\bar{s})\mid}{\rho^{n}} \right\}\right| \leqslant (1-\alpha_{1}) \big| B_{4\rho}(y)\big|,
\end{multline}
for some time level $t_{1}\in (\bar{s} +\delta_{1}\tau_{1},\bar{s}+ \tau_{1})$.

From \eqref{eq5.3} by Theorem \ref{th4.1} with $N=\varepsilon_{1}\dfrac{\nu}{2}N_{\tau_{0}} [\mu(r)]^{-n}\bigg(\dfrac{r}{\rho}\bigg)^{n}$
we have
\begin{equation*}
u(x,t) \geqslant N_{1}:=\sigma_{0}\varepsilon_{1}\frac{\nu}{2}N_{\tau_{0}}\lambda(r)[\mu(r)]^{-n}\bigg(\frac{r}{\rho}\bigg)^{n}, \quad x \in B_{2\rho}(y),
\end{equation*}
for all $t\in(t_{1}+\bar{C}_{1}\rho^{p}N_{1}^{2-p},t_{1}+\bar{C}_{2}\rho^{p}N_{1}^{2-p})$, provided that
$\sigma_{0} N_{1}\geqslant \rho$.\\Since $B_{\rho}(x_{0}) \subset B_{2\rho}(y)$, recalling the definition of $N_{\tau_{0}}$, $N_{1}$ and $r$,  using \eqref{eq1.13} and using the fact that $\lambda_{1}(\rho)= \lambda(\rho)[\mu(\rho)]^{-n}$,  from this we obtain
\begin{equation}\label{eq5.4}
u(x,t)\geqslant \frac{\sigma_{0}}{2^{n+2}}\varepsilon_{1}\nu u_{0} \lambda_{1}(\rho),\quad x\in B_{\rho}(x_{0}),
\end{equation}
for all $t\negthickspace\in \negthickspace(t_{1}\negthickspace +\bar{C}_{1}\rho^{p}N_{1}^{2-p}\negthickspace, t_{1}+\bar{C}_{2}\rho^{p}N_{1}^{2-p})$, provided that $\dfrac{\sigma_{0}}{2^{n+2}}\varepsilon_{1}\nu u_{0} \lambda_{1}(\rho)\negthickspace\geqslant \rho$, which holds by \eqref{eq5.1} if $C$ is chosen to satisfy $C \geqslant 2^{n+2}\sigma_{0}^{-1}\varepsilon_{1}^{-1}\nu^{-1}$.\\
By our choices
$t_{1}\geqslant \bar{s} +\delta_{1}\tau_{1}\geqslant s-\theta\geqslant t_{0}-\rho^{p}(u_{0}\lambda_{1}(\rho))^{2-p}-\theta\quad \text{and}\quad  t_{1}\leqslant \bar{s} +\tau_{1} \leqslant s+\tau_{1}\leqslant t_{0}+ \tau_{1},$
moreover, $\tau_{1}=\rho^{p}\bigg(\dfrac{N_{\tau_{0}}}{2}\lambda(r)\dfrac{\mid E(\bar{s}) \mid}{\rho^{n}}\bigg)^{2-p} \leqslant
\rho^{p}\bigg(\nu\dfrac{N_{\tau_{0}}}{2}\lambda(r)\bigg(\dfrac{r}{\rho}\bigg)^{n}\bigg)^{2-p}
= \rho^{p}\bigg(\dfrac{\nu}{2^{n+2}} u_{0}\lambda_{1}(\rho)\bigg)^{2-p},$
$\theta \leqslant r^{p}N_{\tau_{0}}^{2-p}\leqslant r^{p}\times\\ \times \bigg(\dfrac{u_{0}}{2^{n+1}}\dfrac{\lambda_{1}(\rho)}{\lambda_{1}(r)}\bigg(\dfrac{\rho}{r}\bigg)^{n}\bigg)^{2-p} \leqslant \rho^{p}\bigg(\dfrac{1}{2^{n+1}} u_{0}\lambda_{1}(\rho)\bigg)^{2-p}.$

Therefore, setting $c=\bar{C}_{1}(\sigma_{0}\varepsilon_{1}\nu 2^{-n-2})^{2-p}+\big(\dfrac{\nu}{2^{n+2}}\big)^{2-p}$ and $c_{1}=\bar{C}_{2}(\sigma_{0}\varepsilon_{1}\nu 2^{-n-2})^{2-p}-1-2^{(n+2)(p-2)}$, we obtain that inequality \eqref{eq5.4} holds  for $t_{0}+c \rho^{p}(u_{0}\lambda_{1}(\rho))^{2-p} \leqslant t \leqslant c_{1}\rho^{p}(u_{0}\lambda_{1}(\rho))^{2-p}$, provided that \eqref{eq5.1} is valid and $C \geqslant 2^{n+2}\sigma_{0}^{-1}\varepsilon_{1}^{-1}\nu^{-1}$,  which proves Theorem \ref{th1.2} in the case $\max\limits_{Q_{4r, 4r}(y,\bar{s})} a(x,t)\leqslant 4A \mu(4r) (4r)^{q-p}$.

Now let $\max\limits_{Q_{4r, 4r}(y,\bar{s})} a(x,t)\geqslant 4A \mu(4r) (4r)^{q-p}$ , then there exists $\bar{\rho} \in (r,\rho)$ such that
$\max\limits_{Q_{4\bar{\rho},4\bar{\rho}}(y,\bar{s})} a(x,t)\geqslant 4A \mu(4\bar{\rho}) (4\bar{\rho})^{q-p}$  and $\max\limits_{Q_{8\bar{\rho},8\bar{\rho}}(y,\bar{s})} a(x,t)\negthickspace \leqslant\negthickspace 4A \mu(8\bar{\rho}) (8\bar{\rho})^{q-p} $,
and let $\rho_{0}$ be the maximal number satisfying the above condition. Consider the solution $w(x,t)=w_{r,\frac{1}{2}N_{\tau_{0}}}(x,t,y,\bar{s})$ with $N=\frac{1}{2}N_{\tau_{0}}$, $E=E(\bar{s})$ of the problem \eqref{eq3.6}-\eqref{eq3.8} in $Q_{2}=Q^{+}_{8\rho_{0},8\tau_{2}}(y,\bar{s})$,  $\tau_{2}\negthickspace=\negthickspace\rho_{0}^{p}\bigg(\dfrac{N_{\tau_{0}}}{2}\lambda(r)\dfrac{\mid E(\bar{s}) \mid}{\rho_{0}^{n}}\bigg)^{2-p},$ $E= E(\bar{s})\negthickspace:=\negthickspace\{ B_{r}(y): u(\cdot,\bar{s}) \geqslant \frac{1}{2} N_{\tau_{0}}\lambda(r)\}.$ Inequality \eqref{eq3.16} of Lemma \ref{lem3.2} implies
\begin{equation*}
\left|\left\{B_{4\rho_{0}}(y) :w(\cdot,t_{1}) \leqslant \varepsilon_{1}N_{\tau_{0}}\lambda(r) \frac{\mid E(\bar{s})\mid}{\rho^{n}}\right\}\right| \leqslant
(1-\alpha_{1}) \mid B_{4\rho_{0}}(y) \mid,
\end{equation*}
for some time level $t_{1}\in (\bar{s} +\delta_{2}\tau_{2},\bar{s}+ \tau_{2})$ and the numbers $\varepsilon_{1},\alpha_{1},\delta_{1} \in (0,1)$ depend only on the data.

Similarly to \eqref{eq5.3} by the fact that $u\geqslant w$ on the parabolic boundary of $Q_{2}$ we have
\begin{equation}\label{eq5.5}
\begin{aligned}
\left|\left\{B_{4\rho_{0}}(y) :u(\cdot,t_{1})\leqslant  \varepsilon_{1}\frac{\nu}{2}N_{\tau_{0}}\lambda(r)[\mu(r)]^{-n}\bigg(\frac{r}{\rho_{0}}\bigg)^{n} \right\}\right|\leqslant\\ \leqslant \left|\left\{B_{4\rho_{0}}(y) : u(\cdot,t_{1})\leqslant  \varepsilon_{1}N_{\tau_{0}}\lambda(r)\frac{\mid E(\bar{s})\mid}{\rho_{0}^{n}} \right\}\right| \leqslant \\
\leqslant\negthickspace \left|\left\{B_{4\rho_{0}}(y): w(\cdot,t_{1})\negthickspace\leqslant  \negthickspace \varepsilon_{1}N_{\tau_{0}}\lambda(r)\frac{\mid E(\bar{s})\mid}{\rho_{0}^{n}} \right\}\right|\negthickspace\leqslant\negthickspace(1-\alpha_{1}) \mid B_{4\rho_{0}}(y) \mid,
\end{aligned}
\end{equation}
for some time level $t_{1}\in (\bar{s} +\delta_{2}\tau_{2},\bar{s}+ \tau_{2})$, which by Theorem \ref{th4.1} with $N=\varepsilon_{1}\frac{\nu}{2}N_{\tau_{0}} [\mu(r)]^{-n}\bigg(\dfrac{r}{\rho_{0}}\bigg)^{n}$ implies
\begin{equation}\label{eq5.6}
u(x,t) \geqslant \bar{N}_{1}:=\sigma_{0}\varepsilon_{1}\frac{\nu}{2}N_{\tau_{0}}\lambda(r)[\mu(r)]^{-n}\bigg(\frac{r}{\rho_{0}}\bigg)^{n}, \quad x \in B_{2\rho_{0}}(y),
\end{equation}
for all $t\in(t_{1}+\bar{C}_{1}\rho_{0}^{p}\bar{N}_{1}^{2-p},t_{1}+\bar{C}_{2}\rho_{0}^{p}\bar{N}_{1}^{2-p})$, provided that
$\sigma_{0} \bar{N}_{1}\geqslant \rho_{0}$. We note that this inequality holds if $C\geqslant 2^{n+1}\sigma^{-1}_{0}\varepsilon^{-1}_{1}\nu^{-1}.$

Construct the solution $v=v_{2\rho_{0},\bar{N}_{1}}(x,t,y,t_{2})$ with $N=\sigma_{0}\varepsilon_{1}\frac{\nu}{2}N_{\tau_{0}}$ $[\mu(r)]^{-n}\bigg(\dfrac{r}{\rho_{0}}\bigg)^{n}$ and $E=B_{2\rho_{0}}(y)$ of the problem \eqref{eq3.2}-\eqref{eq3.4} in $\bar{Q}_{1}=Q^{+}_{8\rho,8\bar{\tau}_{1}}(y,t_{2})$, $t_{2}=t_{1}+\bar{C}_{1}\rho_{0}^{p}\bar{N}_{1}^{2-p}$  and  $\bar{\tau}_{1}=\rho^{p}(\bar{N}_{1}(\frac{2\rho_{0}}{\rho})^{n})^{2-p}$. Inequality \eqref{eq3.15} implies
\begin{equation*}
\left|\left\{B_{4\rho}(y) :v(\cdot,t_{3}) \leqslant \varepsilon_{1}\bar{N}_{1} \bigg(\frac{2\rho_{0}}{\rho}\bigg)^{n} \right\} \right|\leqslant
(1-\alpha_{1}) \mid B_{4\rho}(y) \mid,
\end{equation*}
for some time level $t_{3}\in (t_{2} +\delta_{1}\bar{\tau}_{1},t_{2}+ \bar{\tau}_{1})$ with some $\varepsilon_{1},\delta_{1},\alpha_{1} \in (0,1)$ depending only upon the data.

From this, completely similar to \eqref{eq5.3}, \eqref{eq5.4}, by the fact that $u\geqslant v$ on the parabolic boundary of $\bar{Q}_{1}$
and using Theorem \ref{th4.1} we arrive at
\begin{equation*}
u(x,t) \geqslant N_{2}:=\sigma_{0}\varepsilon_{1}\bar{N}_{1}\bigg(\frac{2\rho_{0}}{\rho}\bigg)^{n},\quad x \in B_{2\rho}(y),
\end{equation*}
for all $t\in(t_{3}+\bar{C}_{1}\rho^{p}N_{2}^{2-p},t_{3}+\bar{C}_{2}\rho^{p}N_{2}^{2-p})$, provided that
$\sigma_{0} N_{2}\geqslant \rho$.\\Since $B_{\rho}(x_{0}) \subset B_{2\rho}(y)$, recalling the definition of $N_{0},\bar{N}_{1}$ and $r$, from this, we obtain
\begin{equation}\label{eq5.7}
u(x,t)\geqslant \frac{\sigma^{2}_{0}}{2^{n+2}}\varepsilon^{2}_{1}\nu u_{0} \lambda_{1}(\rho),\quad  x\in B_{\rho}(x_{0}),
\end{equation}
for all $t\in (t_{3}+\bar{C}_{1}\rho^{p}N_{2}^{2-p},t_{3}+\bar{C}_{2}\rho^{p}N_{2}^{2-p})$, provided that $\frac{\sigma^{2}_{0}}{2^{n+2}}\varepsilon^{2}_{1}\nu u_{0} \lambda_{1}(\rho)$ $\geqslant \rho$, which holds by \eqref{eq5.1} if $C$ is chosen to satisfy $C\geqslant 2^{n+2}\sigma_{0}^{-2}\varepsilon_{1}^{-2}\nu^{-1}$.  By our choices
\begin{multline*}
t_{3}\leqslant \bar{s} +\tau_{2}+\bar{\tau}_{1} +
\bar{C}_{1}\rho_{0}^{p} N_{1}^{2-p}\leqslant
t_{0}+\rho_{0}^{p}\left(\frac{N_{\tau_{0}}}{2}\lambda(r) \frac{|E(\bar{s})|}{\rho_{0}^{n}}\right)^{2-p}+\\
+\rho^{p}\left(\sigma_{0}\varepsilon_{1}\frac{\nu}{2} N_{\tau_{0}} \lambda_{1}(r)\left(\frac{r}{\rho_{0}}\right)^{n}\right)^{2-p}\negthickspace\negthickspace+
\bar{C}_{1}\rho_{0}^{p}\left(\sigma_{0}\varepsilon_{1}\frac{\nu}{2} N_{\tau_{0}} \lambda_{1}(r)\left(\frac{2r}{\rho}\right)^{n}\right)^{2-p}\negthickspace\negthickspace\leqslant\\
 \leqslant\negthickspace t_{0}+ \rho^{p}(u_{0}\lambda_{1}(\rho))^{2-p}\left[2^{(n+2)(p-2)}+\left(\sigma_{0}\varepsilon_{1}\frac{\nu}{2^{n+2}}\right)^{2-p}\negthickspace\negthickspace
+\bar{C}_{1}\left(\sigma_{0}\varepsilon_{1}\frac{\nu}{2}\right)^{2-p}\right]
\end{multline*}
 and
$t_{3}\geqslant t_{0}-\rho^{p}(u_{0}\lambda_{1}(\rho))^{2-p}-\theta \geqslant t_{0}- \rho^{p}(u_{0}\lambda_{1}(\rho))^{2-p}(1+ 2^{(n+1)(p-2)}). $

Therefore, setting $c=2^{(n+2)(p-2)} +(\sigma_{0}\varepsilon_{1}\frac{\nu}{2^{n+2}})^{2-p} +\bar{C}_{1}(\sigma_{0}\varepsilon_{1}\frac{\nu}{2})^{2-p} +\bar{C}_{1}(\sigma_{0}^{2}\varepsilon_{1}^{2}\frac{\nu}{2^{n+2}})^{2-p}$ and $c_{1}= \bar{C}_{2}(\sigma_{0}^{2}\varepsilon_{1}^{2}\frac{\nu}{2^{n+2}})^{2-p} -1 -2^{(n+2)(p-2)}$, we obtain that inequality \eqref{eq5.7} holds for all $t_{0}+c\rho^{p}(u_{0}\lambda_{1}(\rho))^{2-p} \leqslant t \leqslant t_{0}+c_{1}\rho^{p}\times\\\times(u_{0}\lambda_{1}(\rho))^{2-p}$, provided that \eqref{eq5.1} is valid and $C \geqslant 2^{n+2}\sigma_{0}^{-2}\varepsilon_{1}^{-2}\nu^{-1}$.  This completes the proof of Theorem \ref{th1.2}.

\vskip3.5mm
\section*{ Acknowledgements} This work is supported by the Grant EFDS-FL2-08 of the found The European Federation of Academies of Sciences and Humanities (ALLEA) and by the Volkswagen Foundation project
”From Modeling and Analysis to Approximation”.

\bigskip

CONTACT INFORMATION

\medskip
Mariia O. Savchenko\\Institute of Applied Mathematics and Mechanics,
National Academy of Sciences of Ukraine, Gen. Batiouk Str. 19, 84116 Sloviansk, Ukraine\\shan$\_$maria@ukr.net

\medskip
Igor I. Skrypnik\\Institute of Applied Mathematics and Mechanics,
National Academy of Sciences of Ukraine, Gen. Batiouk Str. 19, 84116 Sloviansk, Ukraine\\ihor.skrypnik@gmail.com

\medskip
Yevgeniia A. Yevgenieva\\
Institute of Applied Mathematics and Mechanics, National Academy of Sciences of Ukraine, Gen.
Batiouk Str. 19, 84116 Sloviansk, Ukraine\\
yevgeniia.yevgenieva@gmail.com
\end{document}